\documentclass[11pt]{article}

\usepackage[margin=1in]{geometry}

\usepackage{amsmath, amsthm, graphicx, amsfonts, mathtools}
\usepackage{amssymb}
\usepackage[colorlinks=true, allcolors=blue]{hyperref}
\usepackage{tikz-cd}
\usepackage{tikz}
\usepackage{xcolor}
\usetikzlibrary{arrows.meta, calc}
\colorlet{boxgray}{black!70!white}
\usepackage{mathdots}
\usepackage{dsfont}

\usepackage{placeins} 
\usepackage{subcaption}

\DeclareFontFamily{U}{mathb}{\hyphenchar\font45}
\DeclareFontShape{U}{mathb}{m}{n}{
  <5><6><7><8><9><10> gen * mathb
  <10.95> mathb10 <12><14.4><17.28><20.74><24.88> mathb12
}{}

\DeclareSymbolFont{mathb}{U}{mathb}{m}{n}

\DeclareMathSymbol{\boxslash}{\mathord}{mathb}{"6D}
\DeclareMathSymbol{\boxminus}{\mathord}{mathb}{"61}

\newcommand{\zbox}{\Z_\uparrow^4}
\newcommand{\zcross}{\Z^4_\boxminus}
\newcommand{\zdiag}{\Z^4_\boxslash}

\usepackage[
  backend=biber,
  style=alphabetic,
  giveninits=false,
  maxbibnames=99,
  date=year,
  doi=false,
  eprint=true,
  url=false,
  isbn=false
]{biblatex}

\theoremstyle{plain}
\newtheorem{theorem}{Theorem}[section]
\newtheorem{lemma}[theorem]{Lemma}
\newtheorem{proposition}[theorem]{Proposition}
\newtheorem{corollary}[theorem]{Corollary}

\theoremstyle{definition}

\newtheorem{remark}[theorem]{Remark}

\numberwithin{equation}{section}

\newcommand{\R}{\mathbb{R}}
\newcommand{\N}{\mathbb{N}}
\newcommand{\Z}{\mathbb{Z}}
\newcommand{\GT}{\mathrm{GT}}

\DeclareMathOperator{\sh}{sh}
\newcommand{\defeq}{\mathrel{\mathop:}=}
\newcommand{\disteq}{\ {\buildrel d \over =}\ }

\title{Invariance in half-space last passage percolation and directed polymers}
\author{Yunhao Lou}
\date{}

\begin{document}
\maketitle

\begin{abstract}
    We prove various invariance properties
    for half-space geometric and exponential last passage percolation and the half-space log-gamma polymer, analogous to the full-space invariances proved in \cite{dauvergne}. We use the RSK and geometric RSK correspondences to establish a conditional distributional identity between the half-space models and the corresponding full-space models, which allows us to directly transfer the invariance results in the full space to the half space. 
\end{abstract}

\tableofcontents

\section{Introduction}
\label{sec: intro}

\subsection{Background}

Last passage percolation (LPP) has been studied for decades. A landmark result of Baik, Deift, and Johansson \cite{baikDistributionLengthLongest1999} shows that, after suitable centering and scaling, the length of the longest increasing subsequence of a uniformly random permutation converges in distribution to the GUE Tracy--Widom law \cite{tracyLevelspacingDistributionsAiry1994}.
The corresponding fluctuation results for geometric and exponential LPP
were established by Johansson
\cite{johanssonShapeFluctuationsRandom2000}.

While numerous asymptotic results for LPP and other integrable probabilistic models have been obtained since then, 
it is only recently that some surprising underlying invariances within these models have been identified. 
Borodin, Gorin, and Wheeler
\cite{borodin} started from the stochastic colored six-vertex model and showed that in a variety of integrable stochastic systems, the joint distribution of certain observables is invariant under shifts of selected observation points. 
These results were later generalized by Galashin in \cite{galashin}, who discovered the flip-invariance in the stochastic colored six-vertex model. 
Dauvergne \cite{dauvergne} established similar shift invariances and other symmetries in geometric or exponential LPP, as well as in the log-gamma polymer. 
Other works related to invariances in integrable probabilistic models include \cite{kuan, Korotkikh, zhang, batesPermutationInvarianceLastpassage2026}.

Many integrable probabilistic models admit a half-space version. 
For example, we can define a half-space LPP by an environment that is symmetric across the diagonal. 
Some early studies from the perspective of the longest increasing subsequence of random involutions include \cite{baikSymmetrizedRandomPermutations2001, baikAsymptoticsMonotoneSubsequences2001, baikAlgebraicAspectsIncreasing2001}, and some later work includes
\cite{pfaffian, beteaStationaryHalfSpaceLast2020, zengStationaryHalfspaceGeometric2025, beteaHalfspaceAiryStat2022, ferrariTimetimeCovarianceLast2024}.
For some studies on half-space directed polymer models, see
\cite{barraquandIdentityDistributionFullSpace2023,dasConvergenceStationaryMeasures2025, wu, ginsburgPinningDiffusiveFluctuations2024, barraquandStationaryMeasuresLoggamma2023, dasHalfspaceLoggammaPolymer2024}. 
Additional studies on half-space models include 
\cite{heBoundaryCurrentFluctuations2024, sasamotoFluctuationsOneDimensionalPolynuclear2004, barraquand, barraquandHalfSpaceStationaryKardar2020, barraquandStochasticSixvertexModel2018, zhangTASEPHalfspace2025}.

He \cite{he} studied symmetry properties in several half-space integrable models, including the half-space colored stochastic six-vertex model, the half-space beta polymer, and the half-space colored asymmetric simple exclusion process (ASEP). 
However, as discussed in \cite[§1.5]{he}, apart from a shift invariance for single point-to-point passage values in half-space exponential LPP, those results do not cover the integrable half-space LPP or polymer models.

In this paper, we establish some shift invariances in the half-space geometric or exponential LPP and the half-space log-gamma polymer. 
To do this, we will use the \emph{Robinson--Schensted--Knuth} (RSK) correspondence, which is a bijection between matrices with nonnegative integer entries and pairs of semistandard Young tableaux of the same shape. 
For the original work of Knuth, see
\cite{knuthPermutationsMatricesGeneralized1970};
for alternative constructions of the correspondence, see
\cite{fominGeneralizedRobinsonSchenstedKnuthCorrespondence1988, KRATTENTHALER2006404};
for its connection to LPP, see \cite{GREENE1974254}.
See also \cite{dauvergneRSKLastPassage2022} for a recent exposition on how the various forms of the RSK correspondence are applied in both discrete and continuous models of LPP. 
The RSK correspondence can also be formulated in terms of the Schur process, and its half-space version in terms of the Pfaffian Schur process, 
see \cite{pfaffian}. 

For directed polymer models, the geometric RSK (gRSK) correspondence is used instead of the classical RSK;
see \cite{oconnell, corwinTropicalCombinatoricsWhittaker2014a, noumiTropicalRobinsonSchenstedKnuthCorrespondence2004}.
In particular, Section 5 of \cite{oconnell} studies the restriction of the gRSK correspondence to symmetric matrices, which corresponds to the half-space environments considered in this paper. 

\subsection{Setup and notation}

For simplicity, we will only formulate our main results in this introduction in the setting of half-space LPP, since the corresponding results for the half-space log-gamma polymer are almost identical.  

We first introduce the full-space model of LPP, following the notation of \cite{dauvergne}.

Let $M=\{M_x:x\in \Z^2\}$ be an environment of random variables. 

A sequence of points $(\pi(1),\pi(2),\dots, \pi(k))$ in $\Z^2$ is called a \emph{directed path} (or simply a path) if $\pi(i)-\pi(i-1)\in \{(0,1),(1,0)\}$ for each $1 < i\leq k$. 

For a directed path $\pi$ in $\Z^2$, define the weight of $\pi$ in the environment $M$ as 
\[ M(\pi) = \sum_{x\in \pi} M_x.\]

For two points $x=(x_1,x_2),y = (y_1,y_2) \in \Z^2$, write $x \nearrow y$ if $x_1 \leq y_1$ and $x_2 \leq y_2$. Let
\[
\zbox \defeq \{ u = (u^-;u^+) \in \Z^2 \times \Z^2 : u^- \nearrow u^+\}.
\]
Each element $u\in \zbox$ can be seen as a pair of endpoints or a box in $\Z^2$ whose lower-left corner is $u^-$ and upper-right corner is $u^+$. 

For any $u = (u^-;u^+) \in \zbox$, say that a directed path $\pi$ is a $u$-path if $\pi$ goes from $u^-$ to $u^+$. 
Define the last passage value of $u$ by
\[ 
Z_M(u) = \max_\pi M(\pi),
\]
where the max is over all $u$-paths $\pi$. 
For a vector $\mathbf{u} = (u_1,\dots,u_k) \in (\zbox)^k$, we also define its last passage value by
\[ 
Z_M(\mathbf{u}) = \max_\pi \sum_{i=1}^k M(\pi_i)
\]
where the max is taken over all $k$-tuples of disjoint paths $\pi = (\pi_1,\dots, \pi_k)$, where each $\pi_i$ is a $u_i$-path. 
If no such tuple of disjoint paths exists, then the value is undefined and we say that $\mathbf u$ is not \emph{valid}. 

We will identify each $u \in \zbox$ with the box $[u^-_1,u^+_1] \times [u^-_2,u^+_2]$. 
For $u,v\in \zbox$, we say that $u$ and $v$ are disjoint if they are disjoint as boxes, i.e. 
\[
([u^-_1,u^+_1] \times [u^-_2,u^+_2]) \cap ([v^-_1,v^+_1] \times [v^-_2, v^+_2]) = \emptyset.
\]
We say that \emph{$u$ crosses $v$ horizontally} and \emph{$v$ crosses $u$ vertically} if 
\[
[v^-_1,v^+_1] \subseteq [u^-_1,u^+_1] \quad \text{and} \quad 
[u^-_2,u^+_2] \subseteq [v^-_2,v^+_2].
\]

We write 
\begin{align*}
    \mathcal{N} &\defeq \{(u,v)\in \zbox \times \zbox: u \text{ and } v \text{ are disjoint}\} \\
    \mathcal{H} &\defeq \{(u,v) \in \zbox \times \zbox: u \text{ crosses } v \text{ horizontally} \} \\
    \mathcal{V} &\defeq \{(u,v) \in \zbox \times \zbox: u \text{ crosses } v \text{ vertically} \}.
\end{align*} 
For $c \in \Z^2$, let 
\[ T_c u = (u^-+c;u^++c) \in \zbox\]
be the shift of $u$ by $c$. 
We say that $v$ is a shift of $u$ if $v = T_c u$ for some $c\in \Z^2$. 
We also say that the vector $\mathbf{v} = (v_1,\dots, v_k) \in (\zbox)^k$ is a shift of $\mathbf{u} = (u_1,\dots,u_k) \in (\zbox)^k$ if there exists a $c \in \Z^2$ such that $v_i = T_c u_i$ for each $i$.

The LPP model is integrable when the random variables in $M$ are i.i.d. geometric or exponential.
It has been shown by \cite{dauvergne} that in this case there exist some remarkable invariances  
of the joint distribution of last passage values
under certain shifts of the endpoints.

Recently, the half-space settings of many integrable models have attracted much attention, and it is natural to ask if these symmetry results continue to hold in the half space. 
For LPP, a half-space environment is an array $\tilde{M}= \{\tilde{M}_x:x\in \Z^2\}$ satisfying $\tilde{M}_{(i,j)} = \tilde{M}_{(j,i)}$ for all $i,j$. 
Its last passage values $Z_{\tilde{M}}$ are defined as above.
For the half-space model to be integrable, the environment must again be geometric or exponential. 
We take the entries above the diagonal to be i.i.d. geometric (respectively, exponential) random variables and the diagonal entries to be i.i.d. random variables of the same type, possibly with a different parameter.  
From now on in this section, we will always assume that $\tilde{M}$ is one of these integrable half-space environments unless specified otherwise.

We remark that our definition of the half-space environment is slightly different from the usual setup found in the literature. 
Typically, only the environment restricted to the half-plane below the diagonal is considered, i.e. 
\[
\{\tilde{M}_x:x=(x_1,x_2) \in \Z^2, x_1 \geq x_2\},
\]
whereas we extend it to the full plane in a symmetric way. 
In this sense, the model may alternatively be called a symmetric-space model. 
However, since the point-to-point last passage value is the same in the lower half-plane for the two settings, we will continue to use the term \emph{half-space} to refer to this symmetric setup. 
We will discuss this symmetric formulation of the half-space directed polymer model further in Remark~\ref{remark: half-space directed polymer}.

We will mostly deal with passage values for boxes that fully cross the diagonal in the half-space setting. 
Accordingly, we define
\begin{align*}
    \Z^4_\boxminus &= \{u\in \zbox: u^-_1 \geq u^-_2,u^+_1 \leq u^+_2\} , \text{ and}\\
 \Z^4_\boxslash &= \{u\in \zbox: u^-_1=u^-_2,u^+_1=u^+_2\},
\end{align*} 
corresponding to boxes that cross the diagonal vertically and boxes that lie on the diagonal.
We do not need to consider the horizontally crossing boxes since for all $(x_1,y_1;x_2,y_2) \in \zbox$ we have
\[
Z_{\tilde{M}}(x_1,y_1;x_2,y_2) = Z_{\tilde{M}}(y_1,x_1;y_2,x_2).
\]

\subsection{Main results}
\label{subsec: main results}

A variety of half-space integrable models were studied in \cite{he},
where several symmetry results were proved,
and one of their results for the half-space ASEP leads to the following shift invariance of the half-space LPP:

\begin{proposition}
[{\cite[Corollary~1.1]{he}}]
\label{point-to-point shift Jimmy He}
    Let $\tilde M$ be the homogeneous half-space exponential LPP
    environment.
    Suppose $u,u' \in \zcross$ are such that $u'$ is a shift of $u$. Then $ Z_{\tilde{M}}(u) \disteq Z_{\tilde{M}}(u')$. 
\end{proposition}

\begin{figure}
  \centering
  \begin{subfigure}{0.4\linewidth}
    \centering
    \resizebox{\linewidth}{!}{
        \begin{tikzpicture}[x=0.5cm,y=0.5cm,
          >={Latex[length=2.5mm]},
          dot/.style={circle,fill,inner sep=1.2pt}]
        \draw[step=1, thin, gray!50] (0.5, 0.5) grid (14.5, 14.5);
        \draw[boxgray, thick] (3,2) rectangle (8,13);
        \draw[gray,dashed,thick] (0.5,0.5) -- (14.5,14.5);
        \draw[thick, blue]
          (3.5,2.5) node[dot,blue]{} -- (3.5,3.5) -- (4.5,3.5) -- (4.5, 6.5) -- (5.5, 6.5) -- (5.5,10.5) -- (6.5,10.5) -- (6.5,12.5) -- (7.5, 12.5) node[dot,blue]{};
        \node[blue,anchor=center,font = \Large] at (5.5, 2.5){$u$};
    \end{tikzpicture}
    }
    \caption{} 
    \label{fig: prop1.1a}
  \end{subfigure}
  \quad
  \begin{subfigure}{0.4\linewidth}
    \centering
    \resizebox{\linewidth}{!}{
        \begin{tikzpicture}[x=0.5cm,y=0.5cm,
         >={Latex[length=2.5mm]},
        dot/.style={circle,fill,inner sep=1.2pt}]
        \draw[step=1, thin, gray!50] (0.5, 0.5) grid (14.5, 14.5);
        \draw[boxgray, thick] (5,2) rectangle  (10,13);
        \draw[gray, dashed, thick] (0.5,0.5) -- (14.5,14.5);
        \draw[thick, blue]
          (5.5,2.5) node[dot,blue]{} -- (6.5,2.5) -- (6.5,4.5) -- (7.5,4.5) -- (7.5, 8.5) -- (8.5, 8.5) -- (8.5,10.5) -- (9.5,10.5) -- (9.5,12.5) node[dot,blue]{};
        \node[blue,anchor=center,font = \Large] at (7.5, 2.5){$u'$};
    \end{tikzpicture}
    }
    \caption{}
    \label{fig: prop1.1b}
  \end{subfigure}
  \caption{Example of Proposition~\ref{point-to-point shift Jimmy He}, where $u'$ is a shift of $u$ and the last passage values are equal in distribution. The dashed line indicates the diagonal. }
  \label{fig: prop1.1}
\end{figure}
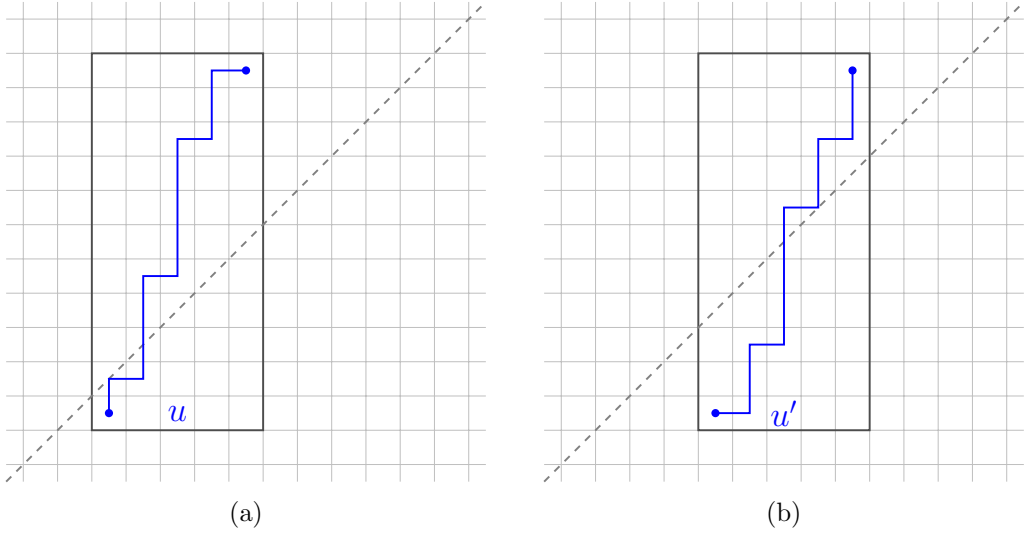

We remark that, although this result appears simple, it is highly nontrivial when the shift is not merely along the diagonal, unlike the corresponding shift of a single box in the full-space environment. 
Indeed, this invariance does not hold in general if the random entries in the half-space environment have arbitrary distributions. 

However, the argument used in \cite{he} only deals with the point-to-point passage values and does not provide the corresponding statement for multi-point passage values. 

Using a different strategy, we obtain a significantly broader class of shifts under which the joint distribution of multi-point passage values is preserved. 
These invariances are similar to those discovered for the full-space model in \cite{dauvergne}, although the symmetric structure of the half-space environment imposes additional constraints on the admissible configurations of boxes, since, for example, we must require that all boxes fully cross the diagonal, in addition to crossing each other. 

In order to avoid more complicated notation, we will not state our main theorems in their strongest forms in this section. 
Instead, we present several representative special cases that capture the essential types of shifts that are allowed.  
First, we extend Proposition~\ref{point-to-point shift Jimmy He} to multi-point passages. 

\begin{theorem}
    \label{multi-point shift intro}
    Suppose $\mathbf{u}=(u_1,\dots,u_k)\in(\zbox)^k$ is valid,
    $\mathbf{u}'=(u'_1,\dots,u'_k)\in(\zbox)^k$ is a shift of
    $\mathbf{u}$, and there exists $v\in\zdiag$ such that
    $(v,u_i),(v,u'_i)\in\mathcal H$ for all $i$.
    Then for the multi-point passage values, we have
    \[ Z_{\tilde{M}}(\mathbf{u}) \disteq Z_{\tilde{M}}(\mathbf{u'}),\]
 and, moreover, the joint distribution of individual point-to-point passage values satisfies
    \[ (Z_{\tilde{M}}(u_1),Z_{\tilde{M}}(u_2),\dots,Z_{\tilde{M}}(u_k)) 
    \disteq
    (Z_{\tilde{M}}(u_1'),Z_{\tilde{M}}(u_2'),\dots,Z_{\tilde{M}}(u_k')).\]
\end{theorem}

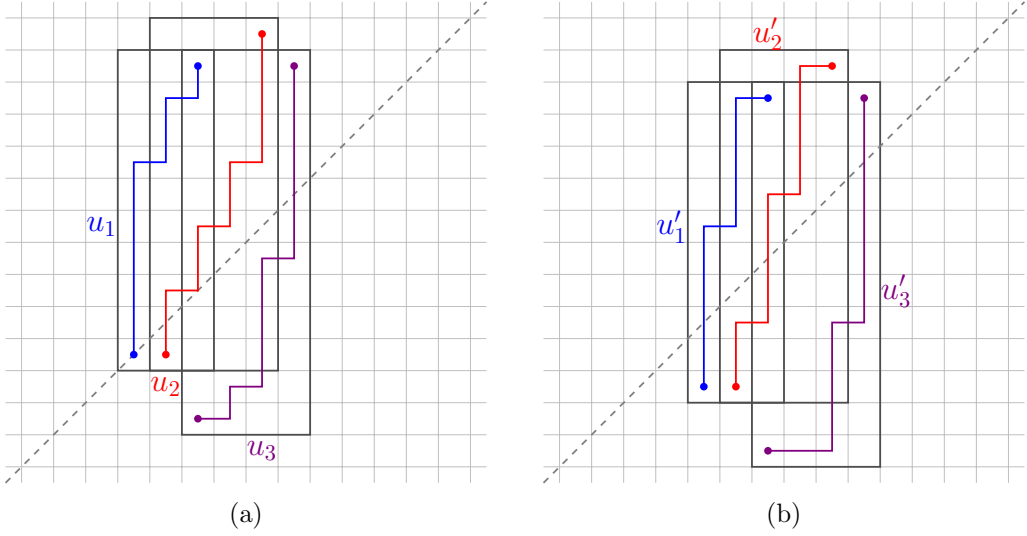
\begin{figure}[h]
  \centering
  \begin{subfigure}{0.4\linewidth}
    \centering
    \resizebox{\linewidth}{!}{
    \begin{tikzpicture}[x=0.5cm,y=0.5cm,
      >={Latex[length=2.5mm]},
      dot/.style={circle,fill,inner sep=1.2pt}]
        \draw[step=1, thin, gray!50] (0.5, 0.5) grid (15.5, 15.5);
        \draw[boxgray, thick] (4,4) rectangle (7,14);
        \draw[boxgray, thick] (5,4) rectangle (9,15);
        \draw[boxgray, thick] (6,2) rectangle (10,14);
        \draw[gray, dashed, thick] (0.5,0.5) -- (15.5,15.5);
        \draw[thick, blue]
          (4.5, 4.5) node[dot,blue]{} -- (4.5, 10.5) -- (5.5, 10.5) -- (5.5,12.5) -- (6.5,12.5) -- (6.5, 13.5) node[dot,blue]{};
        \node[blue,anchor=center,font = \Large] at (3.5, 8.5){$u_1$};
        \draw[thick, red]
          (5.5, 4.5) node[dot,red]{} -- (5.5, 6.5) -- (6.5, 6.5) -- (6.5,8.5) -- (7.5,8.5) -- (7.5, 10.5) -- (8.5,10.5) -- (8.5,14.5) node[dot,red]{};
        \node[red,anchor=center,font = \Large] at (5.5, 3.5){$u_2$};
        \draw[thick, violet]
          (6.5, 2.5) node[dot,violet]{} -- (7.5, 2.5) -- (7.5, 3.5) -- (8.5,3.5) -- (8.5,7.5) -- (9.5, 7.5) -- (9.5,13.5) node[dot,violet]{};
        \node[violet,anchor=center,font = \Large] at (8.5, 1.5){$u_3$};
    \end{tikzpicture}
    }
    \caption{} 
    \label{fig: prop1.2a}
  \end{subfigure}
  \quad
  \begin{subfigure}{0.4\linewidth}
    \centering
    \resizebox{\linewidth}{!}{
    \begin{tikzpicture}[x=0.5cm,y=0.5cm,
      >={Latex[length=2.5mm]},
      dot/.style={circle,fill,inner sep=1.2pt}]
        \draw[step=1, thin, gray!50] (-0.5, 1.5) grid (14.5, 16.5);
        \draw[boxgray, thick] (4,4) rectangle (7,14);
        \draw[boxgray, thick] (5,4) rectangle (9,15);
        \draw[boxgray, thick] (6,2) rectangle (10,14);
        \draw[gray,dashed, thick] (-0.5,1.5) -- (14.5,16.5);
        \draw[thick, blue]
          (4.5, 4.5) node[dot,blue]{} -- (4.5, 9.5) -- (5.5, 9.5) -- (5.5,13.5) -- (6.5, 13.5) node[dot,blue]{};
        \node[blue,anchor=center,font = \Large] at (3.5, 9.5){$u_1'$};
        \draw[thick, red]
          (5.5, 4.5) node[dot,red]{} -- (5.5, 6.5) -- (6.5, 6.5) -- (6.5,10.5) -- (7.5,10.5) -- (7.5, 14.5) -- (8.5,14.5) node[dot,red]{};
        \node[red,anchor=center,font = \Large] at (6.5, 15.5){$u_2'$};
        \draw[thick, violet]
          (6.5, 2.5) node[dot,violet]{} -- (7.5, 2.5) -- (8.5, 2.5) -- (8.5,3.5) -- (8.5,6.5) -- (9.5, 6.5) -- (9.5,13.5) node[dot,violet]{};
        \node[violet,anchor=center,font = \Large] at (10.5, 7.5){$u_3'$};
    \end{tikzpicture}
    }
    \caption{}
    \label{fig: prop1.2b}
  \end{subfigure}
  \caption{Example of Theorem~\ref{multi-point shift intro}. The box $v$ on the diagonal is not shown. }
  \label{fig: prop1.2}
\end{figure}

Here, the auxiliary $v \in \zdiag$ is a common diagonal anchor introduced to control the diagonal-crossing behavior of all the marginal endpoints. 
The existence of a common $v$ cannot in general be replaced by the
condition that the boxes cross the diagonal individually. 
For example, take
$u_1=(0,0;0,2)$, $u_2=(1,1;1,2)$,
and $u_i'=T_{(1,0)}u_i$.
All four boxes cross the diagonal, but
$Z_{\tilde M}(u_1)$ and $Z_{\tilde M}(u_2)$ are independent, whereas
$Z_{\tilde M}(u_1')$ and $Z_{\tilde M}(u_2')$ share the weight
$\tilde M_{(1,2)}=\tilde M_{(2,1)}$.

We can obtain more complicated shift invariances. For simplicity, we demonstrate these only for point-to-point passage values here. 

\begin{theorem}
    \label{intro tower shift}
    Suppose $u_1,\dots, u_k,u_1',\dots, u'_k \in \zcross$ are such that
    $(u_i, u_{i+1}), (u'_i,u'_{i+1}) \in \mathcal{H}$ for $1\leq i < k$, and for every $1\leq i \leq k$ there exists $c_i\in \Z^2$ such that $u'_i = T_{c_i}u_i$ (i.e. $u'_i$ is a shift of $u_i$). 
    Then
    \[ (Z_{\tilde{M}}(u_1),Z_{\tilde{M}}(u_2),\dots,Z_{\tilde{M}}(u_k)) 
    \disteq
    (Z_{\tilde{M}}(u_1'),Z_{\tilde{M}}(u_2'),\dots,Z_{\tilde{M}}(u_k')).\]
\end{theorem} 

\FloatBarrier
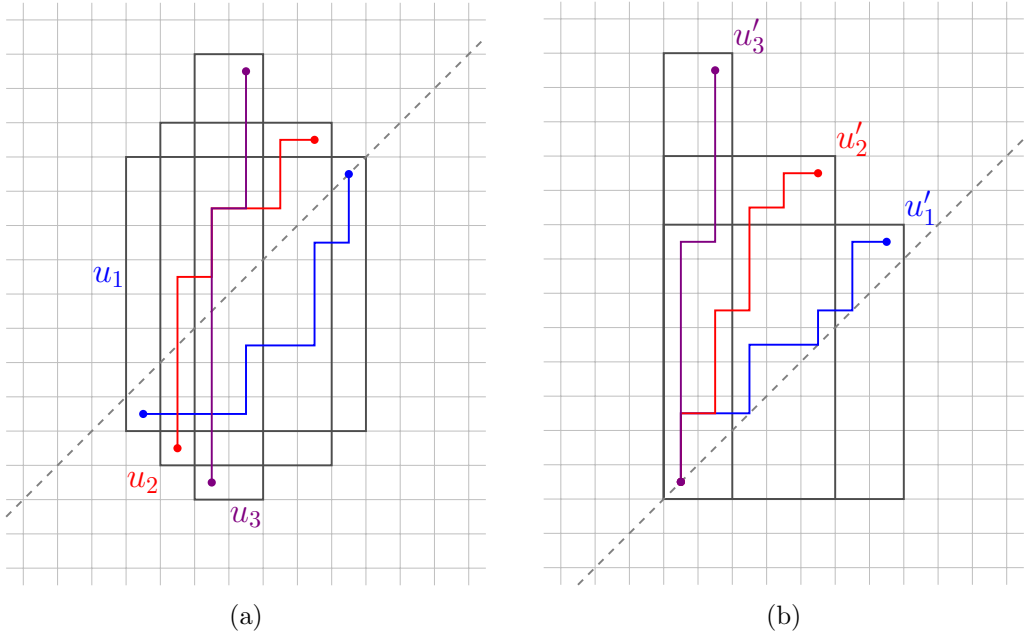
\begin{figure}
  \centering
  \begin{subfigure}{0.4\linewidth}
    \centering
    \resizebox{\linewidth}{!}{
    \begin{tikzpicture}[x=0.5cm,y=0.5cm,
      >={Latex[length=2.5mm]},
      dot/.style={circle,fill,inner sep=1.2pt}]
        \draw[step=1, thin, gray!50] (0.5, -0.5) grid (14.5, 16.5);
        \draw[boxgray, thick] (4,4) rectangle (11,12);
        \draw[boxgray, thick] (5,3) rectangle (10,13);
        \draw[boxgray, thick] (6,2) rectangle (8,15);
        \draw[gray,dashed,thick] (0.5,1.5) -- (14.5,15.5);
        \draw[thick, blue]
          (4.5, 4.5) node[dot,blue]{} -- (7.5, 4.5) -- (7.5, 6.5) -- (9.5,6.5) -- (9.5,9.5) -- (10.5, 9.5) -- (10.5,11.5) node[dot,blue]{};
        \node[blue,anchor=center,font = \Large] at (3.5, 8.5){$u_1$};
        \draw[thick, red]
          (5.5, 3.5) node[dot,red]{} -- (5.5, 8.5) -- (6.5, 8.5) -- (6.5,10.5) -- (8.5,10.5) -- (8.5, 12.5) -- (9.5,12.5) node[dot,red]{};
        \node[red,anchor=center,font = \Large] at (4.5, 2.5){$u_2$};
        \draw[thick, violet]
          (6.5, 2.5) node[dot,violet]{} -- (6.5, 10.5) -- (7.5, 10.5) -- (7.5,14.5) node[dot,violet]{};
        \node[violet,anchor=center,font = \Large] at (7.5, 1.5){$u_3$};
    \end{tikzpicture}
    }
    \caption{}
    \label{fig: prop1.3a}
  \end{subfigure}
  \quad
  \begin{subfigure}{0.4\linewidth}
    \centering
    \resizebox{\linewidth}{!}{
    \begin{tikzpicture}[x=0.5cm,y=0.5cm,
      >={Latex[length=2.5mm]},
      dot/.style={circle,fill,inner sep=1.2pt}]
        \draw[step=1, thin, gray!50] (0.5, 1.5) grid (14.5, 18.5);
        \draw[boxgray, thick] (4,4) rectangle (11,12);
        \draw[boxgray, thick] (4,4) rectangle (9,14);
        \draw[boxgray, thick] (4,4) rectangle (6,17);
        \draw[gray,dashed, thick] (1.5,1.5) -- (14.5,14.5);
        \draw[thick, blue]
          (4.5, 4.5) node[dot,blue]{} -- (4.5, 6.5) -- (6.5, 6.5) -- (6.5,8.5) -- (8.5,8.5) -- (8.5, 9.5) -- (9.5,9.5) -- (9.5,11.5) -- (10.5,11.5) node[dot,blue]{};
        \node[blue,anchor=center,font = \Large] at (11.5, 12.5){$u_1'$};
        \draw[thick, red]
          (4.5, 4.5) node[dot,red]{} -- (4.5, 6.5) -- (5.5, 6.5) -- (5.5,9.5) -- (6.5,9.5) -- (6.5, 12.5) -- (7.5,12.5) -- (7.5,13.5) -- (8.5,13.5) node[dot,red]{};
        \node[red,anchor=center,font = \Large] at (9.5, 14.5){$u_2'$};
        \draw[thick, violet]
          (4.5, 4.5) node[dot,violet]{} -- (4.5, 11.5) -- (5.5, 11.5) -- (5.5,16.5) node[dot,violet]{};
        \node[violet,anchor=center,font = \Large] at (6.5, 17.5){$u_3'$};
    \end{tikzpicture}
    }
    \caption{} 
    \label{fig: prop1.3b}
  \end{subfigure}
  \caption{Example of Theorem~\ref{intro tower shift}. }
  \label{fig: prop1.3}
\end{figure}

\begin{theorem}
    \label{intro disjoint shift}
    
    Suppose $v,u_1,\dots, u_k,u_1',\dots, u'_k \in \zcross$ are such that
    $(v, u_i), (v,u'_i) \in \mathcal{H}$ for all $i$, 
    $(u_i,u_j),(u'_i,u'_j) \in \mathcal{N}$ for all $i\neq j$, 
    and there exists $c_i\in \Z^2$ such that $u'_i = T_{c_i}u_i$ (i.e. $u'_i$ is a shift of $u_i$). 
    Then
    \[ (Z_{\tilde{M}}(v),Z_{\tilde{M}}(u_1),\dots,Z_{\tilde{M}}(u_k)) 
    \disteq
    (Z_{\tilde{M}}(v),Z_{\tilde{M}}(u_1'),\dots,Z_{\tilde{M}}(u_k')).\]
\end{theorem}

\begin{figure}
  \centering
  \begin{subfigure}{0.4\linewidth}
    \centering
    \resizebox{\linewidth}{!}{
    \begin{tikzpicture}[x=0.5cm,y=0.5cm,
      >={Latex[length=2.5mm]},
      dot/.style={circle,fill,inner sep=1.2pt}]
        \draw[step=1, thin, gray!50] (1.5, 0.5) grid (18.5, 20.5);
        \draw[boxgray, thick] (5,5) rectangle (15,16);
        \draw[boxgray, thick] (7,4) rectangle (10,19);
        \draw[boxgray, thick] (11,2) rectangle (13,17);
        \draw[gray,dashed,very thick] (1.5,1.5) -- (18.5,18.5);
        \draw[thick, blue]
          (5.5, 5.5) node[dot,blue]{} -- (8.5, 5.5) -- (8.5, 6.5) -- (9.5, 6.5) -- (9.5, 9.5) -- (11.5, 9.5) -- (11.5, 12.5) -- (13.5, 12.5) -- (14.5, 12.5) -- (14.5, 14.5) -- (14.5, 15.5) node[dot,blue]{};
        \node[blue,anchor=center,font = \Large] at (4.5, 10.5){$v$};
        \draw[thick, red]
          (7.5, 4.5) node[dot,red]{} --(7.5, 10.5) --(8.5, 10.5) --(8.5, 13.5) --(9.5, 13.5) --(9.5, 16.5) --(9.5, 18.5) node[dot,red]{};
        \node[red,anchor=center,font = \Large] at (8.5, 3.5){$u_1$};
        \draw[thick, violet]
          (11.5, 2.5) node[dot,violet]{} --(11.5, 6.5) --(12.5, 6.5) --(12.5, 8.5) --(12.5, 11.5) --(12.5, 14.5) --(12.5, 16.5) node[dot,violet]{};
        \node[violet,anchor=center,font = \Large] at (13.5, 3.5){$u_2$};
    \end{tikzpicture}
    }
    \caption{}
    \label{fig: prop1.4a}
  \end{subfigure}
  \quad
  \begin{subfigure}{0.4\linewidth}
    \centering
    \resizebox{\linewidth}{!}{
    \begin{tikzpicture}[x=0.5cm,y=0.5cm,
      >={Latex[length=2.5mm]},
      dot/.style={circle,fill,inner sep=1.2pt}]
        \draw[step=1, thin, gray!50] (1.5, 0.5) grid (18.5, 20.5);
        \draw[boxgray, thick] (5,5) rectangle (15,16);
        \draw[boxgray, thick] (6,3) rectangle (8,18);
        \draw[boxgray, thick] (11,3) rectangle (14,18);
        \draw[gray,dashed,very thick] (1.5,1.5) -- (18.5,18.5);
        \draw[thick, blue]
          (5.5, 5.5) node[dot,blue]{} --(5.5, 7.5) --(7.5, 7.5) --(7.5, 10.5) --(9.5, 10.5) --(9.5, 12.5) --(11.5, 12.5) --(11.5, 14.5) --(13.5, 14.5) --(13.5, 15.5) --(14.5, 15.5) node[dot,blue]{};
        \node[blue,anchor=center,font = \Large] at (4.5, 10.5){$v$};
        \draw[thick, violet]
          (6.5, 3.5) node[dot,violet]{} --(6.5, 12.5) --(7.5, 12.5) --(7.5, 15.5) --(7.5, 17.5) node[dot,violet]{};
        \node[violet,anchor=center,font = \Large] at (6.5, 2.5){$u_2'$};
        \draw[thick, red]
          (11.5, 3.5) node[dot,red]{} --(11.5, 6.5) --(12.5, 6.5) --(12.5, 9.5) --(12.5, 11.5) --(13.5, 11.5) --(13.5, 15.5) --(13.5, 17.5) node[dot,red]{};
        \node[red,anchor=center,font = \Large] at (12.5, 2.5){$u_1'$};
    \end{tikzpicture}
    }
    \caption{}
    \label{fig: prop1.4b}
  \end{subfigure}
  \caption{Example of Theorem~\ref{intro disjoint shift}. }
  \label{fig: prop1.4}
\end{figure}
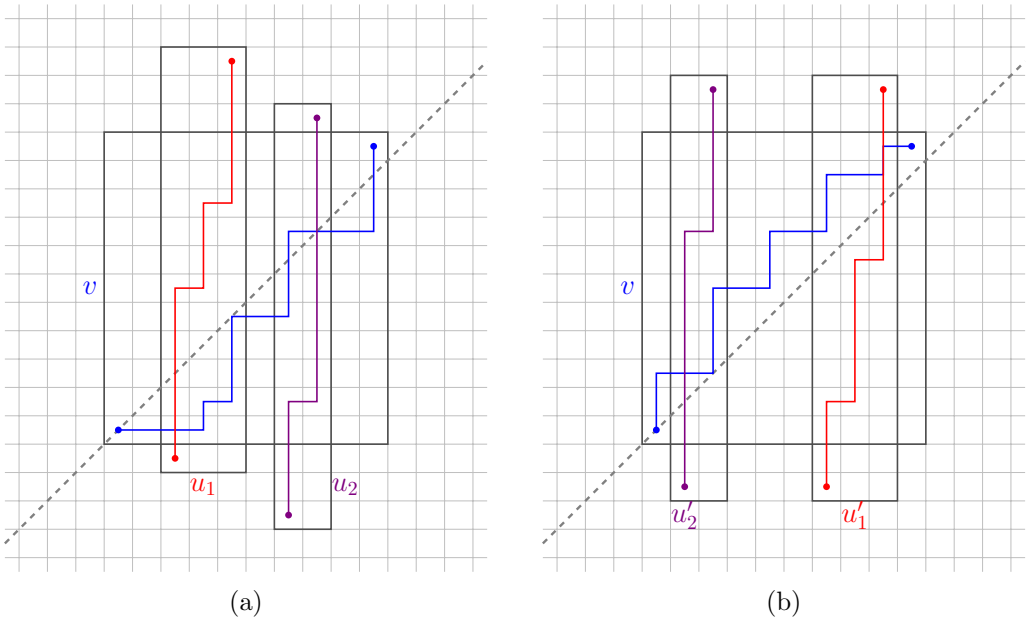

These are all simplified cases of our main shift-invariance theorems for homogeneous half-space environments, Theorems~\ref{tower shift} and \ref{disjoint shift}.
In the full statements, each individual box appearing above may be replaced by a subset of $\zcross$ and certain multi-point passage values are also preserved jointly with the point-to-point passage values. 

Our method also yields shift- and parameter-permuting invariances in inhomogeneous environments. 
For $\alpha$ a bi-infinite sequence of real numbers and $c \in \R$, define the inhomogeneous half-space exponential environment by assigning an independent $\mathrm{Exp}(c+\alpha_i)$ random variable to each diagonal position $(i,i)$,
an independent $\mathrm{Exp}(\alpha_i+\alpha_j)$ random variable to each off-diagonal position $(i,j)$ with $i<j$, 
and defining the remaining off-diagonal entries $(i,j)$ with $i>j$ by symmetry.
The inhomogeneous half-space geometric environment is defined analogously, with diagonal entries having geometric distributions with probability mass $(1-c\alpha_i)(c\alpha_i)^k$ and off-diagonal entries having mass $(1-\alpha_i \alpha_j)(\alpha_i\alpha_j)^k$, where each $\alpha_i$ and $c$ are positive. 
For the half-space definitions, we require $\alpha_i+\alpha_j >0$ for all $i\neq j$ and $c+\alpha_i > 0$ for all $i \in \Z$ in the exponential case, and $\alpha_i\alpha_j \in (0,1)$ for all $i\neq j$ and $c\alpha_i \in (0,1)$ for all $i\in \Z$ in the geometric case.

From now on in this section, we will assume that $\tilde{M}_{\alpha,c}$ is an integrable inhomogeneous half-space environment of one of the types described above (either exponential or geometric) with parameters $\alpha$ and $c$, and that $\tilde{M}_{\alpha',c}$ is another environment of the same type defined using parameters $\alpha'$ and $c$.

Our next result tells us that if a tower of boxes undergoes a shift that preserves the tower structure and if the parameters are permuted so that the parameters associated with each box are preserved under the shift, then the joint distribution of passage values across these boxes is also preserved.

For $u \in \zbox$, we define, for convenience in the statement of the next result, 
\[
I(u) = [u^-_1,u^+_1] \cap \Z,
\qquad 
J(u) = [u^-_2, u^+_2] \cap \Z. 
\]

\begin{theorem}
\label{intro inhomogeneous shift}
    Let $u_1,\dots, u_k, u_1',\dots, u'_k \in \zcross$ be as in Theorem~\ref{intro tower shift}.
    Assume $\alpha_i, \alpha_i' \in (0,1)$ for every $i\in \Z$ in the geometric case, and $\alpha_i, \alpha_i' >0 $ for every $i\in \Z$ in the exponential case. 
    Suppose that the sequences $\alpha, \alpha'$ satisfy the property that
    $(\alpha_i: i\in I(u_j))$ is a permutation of $(\alpha'_i:i\in I(u_j'))$ and that 
    $(\alpha_i: i\in J(u_j))$ is a permutation of $(\alpha'_i:i\in J(u_j'))$
    for all $1\leq j\leq k$. 
    Then, for every $c$ for which both
    $\tilde M_{\alpha,c}$ and
    $\tilde M_{\alpha',c}$ are well-defined,
    \[ (Z_{\tilde{M}_{\alpha,c}}(u_1),
    Z_{\tilde{M}_{\alpha,c}}(u_2),\dots,
    Z_{\tilde{M}_{\alpha,c}}(u_k)) 
    \disteq
    (Z_{\tilde{M}_{\alpha',c}}(u_1'),
    Z_{\tilde{M}_{\alpha',c}}(u_2'),
    \dots,Z_{\tilde{M}_{\alpha',c}}(u_k')).\]
\end{theorem}

This result is a simplified version of Theorem~\ref{inhomogeneous shift and permute}, which establishes invariances in half-space inhomogeneous models. 

\begin{remark}
Some related parameter-permuting invariances for inhomogeneous half-space environments have been established in \cite[Lemmas~2.3 and~2.6]{dauvergne-zhang}.
Their results keep the boxes fixed and also apply to boxes that do not fully cross the diagonal.
Our inhomogeneous shift result requires the boxes to fully cross the diagonal but allows the boxes' positions to be shifted together with appropriate permutations of the parameters.
\end{remark}

In addition to the LPP invariances stated above, we will show that similar invariance results also hold for the half-space log-gamma polymer, where the last passage value is replaced by the polymer partition function.

\subsection{Idea of proof}

Rather than proving the half-space invariances directly, our approach is to establish an identity in conditional distributions between the full-space and half-space models. 
This identity effectively allows us to replace any symmetric part of a region in the half-space environment with a corresponding region in the full-space environment. 
As a consequence, we can then directly transfer invariance results from the full-space to the half-space setting.

A homogeneous version of the identity in conditional distributions can be formulated as follows.
Recall that we say a vector $\mathbf u \in (\zbox)^k$ is valid if there exist disjoint paths $\pi_1,\dots, \pi_k$ such that each $\pi_i$ is a $u_i$-path. 
For $v \in \zbox$, let $\bar{V}(v)$ be the set of all valid vectors $\mathbf u$ in some $(\zbox)^k$ where each $u_i$ starts on $\{v^-_1, v^-_1+1,\dots, v^+_1\}\times \{v^-_2\}$ and ends on
    $\{v^-_1, v^-_1+1,\dots, v^+_1\}\times \{v^+_2\}$,
    i.e. $u_i$ is contained in $v$ as a box and has the same height as $v$. 
Moreover, let $D(v)$ be the set of vectors of the form
\[
v^k \defeq \bigl((v^-;v^+ - (k-1,0)),(v^-+(1,0);v^+-(k-2,0)),\dots, (v^- + (k-1,0);v^+)\bigr)
\in (\zbox)^k,
\]
where
$1\leq k\leq\min\{v^+_1-v^-_1+1,\;v^+_2-v^-_2+1\}$.
Note that $D(v)\subseteq\bar V(v)$.

\begin{proposition}
\label{intro bridge}
    Suppose $M$ is an integrable full-space LPP environment and $\tilde M$ is a half-space LPP environment of the same type (either geometric or exponential), and that the off-diagonal entries of $\tilde M$ have the same distribution as the entries of $M$. 
    Then, for each $v \in \zdiag$, we have the following identity in conditional laws:
    \[
    \mathcal L\!\left(
        Z_M|_{\bar V(v)}
        \,\middle|\,
        Z_M|_{D(v)}
    \right)
    =
    \mathcal L\!\left(
        Z_{\tilde M}|_{\bar V(v)}
        \,\middle|\,
        Z_{\tilde M}|_{D(v)}
    \right).
    \]
\end{proposition}

Here and throughout the paper, $\mathcal L(A\mid B)$ denotes a version
of the regular conditional distribution of $A$ given $B$, viewed as a
probability kernel. Whenever we assert
$\mathcal L(A\mid B)=\mathcal L(A'\mid B')$,
we mean that $B$ and $B'$ take values in the same measurable space
and that there exists a probability kernel on this space that is a
version of both conditional distributions. In the applications below,
such a common kernel is supplied by the explicit RSK or gRSK
disintegrations.

The proof of this proposition proceeds by comparing explicit expressions for the relevant push-forward probability measures of the two environments under the RSK correspondence, and by using the isometry property of the RSK.

To illustrate how this identity is applied, we sketch how Proposition~\ref{point-to-point shift Jimmy He} can be proved from it together with known invariances in the full-space environment,
where Theorems~\ref{multi-point shift intro}--\ref{intro inhomogeneous shift} are proved following the same idea. 

\begin{proof}
    [Proof sketch of Proposition~\ref{point-to-point shift Jimmy He}] 
    It suffices to consider the case in which $u = (x_1,y_1;x_2,y_2) \in \zcross$ and $u' = (x_1,y_1+c;x_2,y_2+c)\in \zcross$.
    We then define $v = (x_1,x_1;x_2,x_2) \in \zdiag.$
    
    Observe that $Z_M(u)$ and $Z_M(u')$ can each be expressed as functions of two independent components, the collection of passage values $Z_M|_{\bar{V}(v)}$ and the set of environment variables $\{M_x:x\in u \cup u'\setminus v\}$.
    That is, we may write
    \[
    Z_{{M}}(u) = \varphi(Z_M|_{\bar{V}(v)}, \{M_x:x\in u \cup u'\setminus v\})
    \]
    and  
    \[
    Z_{{M}}(u') = \varphi'(Z_M|_{\bar{V}(v)}, \{M_x:x\in u \cup u'\setminus v\})
    \]
    for some deterministic functions $\varphi$ and $\varphi'$. 
    It is not hard to write them out explicitly (see the proof of \cite[Proposition~3.1]{dauvergne}), but this is not needed for us. 
    Replacing the full-space environment $M$ with the half-space environment $\tilde M$, we can similarly write
    \[
    Z_{\tilde{M}}(u) = \varphi(Z_{\tilde{M}}|_{\bar{V}(v)}, \{\tilde{M}_x:x\in u \cup u'\setminus v\})
    \]
    and 
    \[
    Z_{\tilde{M}}(u') = \varphi'(Z_{\tilde{M}}|_{\bar{V}(v)}, \{\tilde{M}_x:x\in u \cup u'\setminus v\})
    \]
    using the same functions. 

    Now, using the full-space shift invariance as proved in \cite{dauvergne}, we have equality in joint distributions
    \[
    (Z_M|_{D(v)}, Z_M(u)) \disteq (Z_M|_{D(v)}, Z_M(u'))
    \]
    and thus in conditional laws
    \[
    \mathcal L \! \left(Z_M(u) \,\middle|\, Z_M|_{D(v)} \right) 
    =
    \mathcal L \! \left( Z_M(u') \,\middle|\, Z_M|_{D(v)} \right).
    \]
    
    Since $Z_M|_{D(v)}$ and $Z_M|_{\bar{V}(v)}$ are determined by a set of environmental variables completely independent from $\{M_x:x \in u \cup u' \setminus v\}$, the decoupling above gives
    \begin{equation}
    \label{eq: decoupled conditional distributions}
    \mathcal L \! \left(
    \varphi(Z_M|_{\bar{V}(v)}, \{M_x:x\in u \cup u'\setminus v\}) 
    \,\middle|\,
    Z_M|_{D(v)} \right)
    =
    \mathcal L \! \left(
    \varphi'(Z_M|_{\bar{V}(v)}, \{M_x:x\in u \cup u'\setminus v\})
    \,\middle|\,
    Z_M|_{D(v)} \right).
    \end{equation}
    
    As a consequence of Proposition~\ref{intro bridge} and the fact that 
    \[
    \{ M_x: x\in u\cup u' \setminus v\} \disteq 
    \{\tilde M_x: x\in u\cup u' \setminus v\},
    \]
    we have
    \[
    \mathcal L \! \left(
    (Z_M|_{\bar V(v)}, \{M_x: x\in u\cup u' \setminus v\})
    \,\middle|\,
    Z_M |_{D(v)}\right)
    =
    \mathcal L \! \left(
    (Z_{\tilde{M}}|_{\bar V(v)}, \{\tilde M_x: x\in u\cup u' \setminus v\})
    \,\middle|\,
    Z_{\tilde{M}} |_{D(v)}\right).
    \]
    
    By Lemma~\ref{lemma: conditioning vector absolute continuity},
    the law of $Z_{\tilde M}|_{D(v)}$ is absolutely continuous
    with respect to the law of $Z_M|_{D(v)}$.
    Combining this with Proposition~\ref{intro bridge} and
    \eqref{eq: decoupled conditional distributions}, we obtain
    \[
    \mathcal L\!\left(
    \varphi\!\left(Z_{\tilde M}|_{\bar V(v)},\{\tilde M_x:x\in u\cup u'\setminus v\}\right)
    \,\middle|\,
    Z_{\tilde M}|_{D(v)}
    \right)
    =
    \mathcal L\!\left(
    \varphi' \!\left(Z_{\tilde M}|_{\bar V(v)},\{\tilde M_x:x\in u\cup u'\setminus v\}\right)
    \,\middle|\,
    Z_{\tilde M}|_{D(v)}
    \right),
    \]
    which means
    \[
    \mathcal L\!\left(Z_{\tilde M}(u)\,\middle|\,Z_{\tilde M}|_{D(v)}\right)
    =
    \mathcal L\!\left(Z_{\tilde M}(u')\,\middle|\,Z_{\tilde M}|_{D(v)}\right).
    \]
    Therefore,
    \[
    Z_{\tilde{M}}(u) \disteq Z_{\tilde{M}}(u').
    \]
\end{proof}

The proofs in the half-space log-gamma polymer setting are almost identical, with the only difference being that the RSK correspondence is replaced by its geometric version in order to obtain the identity in conditional laws of partition functions. 

\subsection{Organization of the text}

In Section~\ref{sec: preliminary}, we provide some preliminary results on the RSK and gRSK correspondences, and demonstrate how they are applied in the geometric LPP and log-gamma polymer settings to give the conditional distributional identities that we need. 
In Section~\ref{main section}, we use the identities from Section~\ref{sec: preliminary} to prove our main half-space invariance results and establish a related conditional independence property.
In Appendix~\ref{appendix}, we prove a technical lemma of the RSK correspondence on symmetric matrices that is used in Section~\ref{sec: preliminary}. 

\subsection*{Acknowledgments}

We are deeply grateful to Lingfu Zhang for his mentorship, invaluable
insights, and continued support throughout this project. 
This project was primarily conducted during the summer of 2025 as part
of the Caltech Summer Undergraduate Research Fellowships (SURF) program
and was partially supported by NSF grant DMS-2505625.
We thank Sayan Das, Jimmy He, and Ran Tao for helpful discussions. 
We are especially grateful to
Jimmy He for pointing out the conditional independence property presented in
Subsection~\ref{subsec: half-space conditional independence} and encouraging its
inclusion here,
to Ran Tao for helpful comments on its proof,
and to Sayan Das for drawing our attention to the parameter-permuting
invariances in \cite{dauvergne-zhang} and suggesting a comparison with
our results.

\subsection*{Statement on the use of AI}

Generative AI tools were used to assist with drafting portions of the
limiting argument in the proof of Theorem~\ref{bridge},
the statement
and proof of Lemma~\ref{lemma: conditioning vector absolute continuity},
the proof of
Lemma~\ref{lemma: conditional independence under marginal change},
and the auxiliary full-space construction in
Remark~\ref{remark: extension of parameter range}.
All other mathematical results and proofs were developed and written
by the author.

\section{Preliminary results} \label{sec: preliminary}

In this section, we present some known combinatorial results on the RSK and geometric RSK correspondences, and apply them to the geometric LPP and log-gamma polymer environments, in both the full-space and half-space settings.

\subsection{The RSK correspondence} \label{subsec: RSK}

First, we recall some combinatorial objects. 
For a nonnegative integer $m$, we say that a weakly decreasing sequence of positive integers $\lambda = (\lambda_1,\lambda_2,\dots, \lambda_k)$ is a \emph{partition} of $m$ if $\sum_{i=1}^k \lambda_i = m$. 
Equivalently, we may allow $\lambda$ to have zeros at the end so it becomes an infinite sequence with finitely many nonzero terms at the beginning.  
The number $m$ is called the \emph{size} of $\lambda$, denoted $|\lambda|$.
The \emph{Young diagram} of $\lambda$ is the set $\{ (i,j)\in \Z^2: 1\leq j\leq k, \ 1\leq i\leq \lambda_j\}$, which we usually think of as a configuration of $|\lambda|=m$ boxes with $\lambda_j$ boxes on the $j$-th row, aligned on the left, and with rows numbered from top to bottom. 

For $\lambda$ a partition of $m$,
a \emph{semistandard Young tableau} of \emph{shape} $\lambda$ is a filling of the Young diagram of $\lambda$ with positive integers such that the entries are strictly increasing along each column and weakly increasing along each row. 

\begin{figure}[h]
    \centering
    \begin{subfigure}{0.4\linewidth}
    \centering
    \begin{tikzpicture}[x=0.7cm,y=-0.7cm]
    \foreach \r/\len in {1/4,2/3,3/1}{
        \foreach \c in {1,...,\len}{
            \draw (\c-1,\r-1) rectangle ++(1,1);
        }
    }
    \node at (0.5,0.5) {$1$};
    \node at (1.5,0.5) {$1$};
    \node at (2.5,0.5) {$2$};
    \node at (3.5,0.5) {$3$};
    \node at (0.5,1.5) {$2$};
    \node at (1.5,1.5) {$3$};
    \node at (2.5,1.5) {$3$};
    \node at (0.5,2.5) {$3$};
    \end{tikzpicture}
    \end{subfigure}
    \begin{subfigure}{0.4\linewidth}
        \centering
        \begin{tikzpicture}
          \def\dx{1.6}
          \def\dy{0.8}
          \node at (0, 2*\dy) {$2$};
          \node at (-0.5*\dx, 1*\dy) {$1$};
          \node at ( 0.5*\dx, 1*\dy) {$3$};
          \node at (-1.0*\dx, 0*\dy) {$1$};
          \node at ( 0.0*\dx, 0*\dy) {$3$};
          \node at ( 1.0*\dx, 0*\dy) {$4$};
        \end{tikzpicture}
    \end{subfigure}
    
    \caption{A semistandard Young tableau of shape $(4,3,1)$ and the corresponding Gelfand--Tsetlin pattern.}
    \label{fig:ssyt}
\end{figure}

There is a natural correspondence between semistandard Young tableaux with filling integers at most $n$ and Gelfand--Tsetlin patterns of $n$ rows. 
A \emph{Gelfand--Tsetlin pattern} is a triangular array of nonnegative integers of the following form
\[
G \;=\;
\begin{array}{ccccccc}
  & &   & g_{11} &   & &   \\[6pt]
  & & g_{22} &   & g_{21} & &   \\[6pt]
& \iddots &   & \vdots &   & \ddots &  \\[6pt]
g_{nn} & & &  \cdots & & &  g_{n1} 
\end{array}
\]
such that $g_{i+1,j+1}\leq g_{ij} \leq g_{i+1,j}$ for all $1\leq j\leq i < n$.
The correspondence goes as follows: $g_{ij}$ is the number of entries less than or equal to $i$ on the $j$-th row in the semistandard Young tableau. 
The last row of the Gelfand--Tsetlin pattern (in right-to-left order) then corresponds to the shape of the semistandard Young tableau, which is a partition. 
Thus, we define the last row of a Gelfand--Tsetlin pattern $G$, read right-to-left as $(g_{n1},\dots, g_{nn})$, to be its \emph{shape}, and denote it as $\sh(G)$. 
We write $\GT_n$ for the set of Gelfand--Tsetlin patterns of $n$ rows, and $\GT_n(\lambda)$ for those with the shape $\lambda$. The subscripts may be omitted when there is no ambiguity. 

The classical RSK correspondence provides a bijection between $n \times n$ matrices with nonnegative integer entries and pairs of semistandard Young tableaux of the same shape with filling integers not exceeding $n$, or equivalently pairs of Gelfand--Tsetlin patterns of the same shape with $n$ rows. 
To be consistent with the discussion of geometric RSK later which involves more general patterns, we proceed using Gelfand--Tsetlin patterns here. 

To avoid repetition, we assume all matrices mentioned to be $n\times n$ nonnegative-integer-valued and all Gelfand--Tsetlin patterns to be of $n$ rows in Subsections~\ref{subsec: RSK} to \ref{subsec: half space geom lpp}.
We denote the input matrix for the RSK as $M = (m_{ij}) \in (\Z_{\geq 0})^{n \times n}$, and its corresponding pair of Gelfand--Tsetlin patterns as $(P,Q)$, with entries indexed in the same manner as $G$ above. 

There are multiple ways of defining the RSK map, including using insertion moves (as in \cite{stanley}, also given in Appendix~\ref{appendix}) or diagrammatic approaches (as in \cite{KRATTENTHALER2006404}).
Here we adopt a definition using the multi-point last passage values, whose equivalence with the classical definitions comes from the well-known Greene's theorem \cite{GREENE1974254}. 
Under this definition, it is convenient to interpret the matrix $M$ as placing a weight $m_{ij}$ at each lattice point $(i,j)\in [n]^2$. 

For $1\leq k \leq a\wedge b$, let $\Pi_{a,b}^k$ be the set of $k$-tuples of non-intersecting directed paths starting at the points $(1,1),(1,2), \dots, (1,k)$ and ending at the points $(a,b-k+1),(a,b-k+2), \dots, (a,b)$, respectively. 
Now, we define the entries of the Gelfand--Tsetlin pattern $P = (p_{ij})_{1\leq j\leq i \leq n}$ by requiring that,
for $1\leq k \leq r \leq n$, 
\begin{equation}
\label{eq: RSK def P}
    \sum_{i=1}^k p_{ri} = \max_{(\pi_1,\pi_2,\dots,\pi_k)\in \Pi_{n,r}^k} \, 
    \sum_{(i,j) \in \pi_1 \cup \pi_2 \cup \dots \cup \pi_k}
    m_{ij},
\end{equation}
where the right hand side is called the $k$-point last passage value from $(1,1)$ to $(n,r)$. 
Similarly, we define the Gelfand--Tsetlin pattern $Q = (q_{ij})_{1\leq j \leq i \leq n}$ by requiring that, for $1\leq k \leq r \leq n$,
\begin{equation}
\label{eq: RSK def Q}
    \sum_{i=1}^k q_{ri}= \max_{(\pi_1,\pi_2,\dots,\pi_k)\in \Pi_{r,n}^k} \, 
    \sum_{(i,j) \in \pi_1 \cup \pi_2 \cup \dots \cup \pi_k}
    m_{ij}.
\end{equation}

Taking $r=n$, it is immediate that the last rows of $P$ and $Q$ are identical. 
Also, since $\Pi_{a,b}^{a\wedge b}$ contains only one set of non-intersecting directed paths using all of $\{(i,j):1\leq i \leq a, \ 1\leq j \leq b\}$, 
we have that for each $1\leq r \leq n$,
\[ \sum_{i=1}^r p_{ri} = \sum_{j=1}^r \sum_{i=1}^n m_{ij},\qquad  \sum_{i=1}^r q_{ri} = \sum_{i=1}^r \sum_{j=1}^n m_{ij}. \]
Hence, for each $1\leq r\leq n$, we have
\begin{equation} 
\label{eq: row sum}
\sum_{i=1}^n m_{ir} = \sum_{i=1}^r p_{ri} - \sum_{i=1}^{r-1} p_{r-1,i},
\end{equation}
and
\begin{equation} 
\label{eq: col sum}
\sum_{i=1}^n m_{ri} = \sum_{i=1}^r q_{ri} - \sum_{i=1}^{r-1} q_{r-1,i},
\end{equation}
where the last terms are regarded as 0 for $r=1$. 

Although bijectivity is not apparent from this definition of the RSK map, the symmetry property follows more directly:

\begin{theorem} \label{RSK symmetry}
    If the matrix $M = (m_{ij}) \in (\Z_{\geq 0})^{n \times n}$ maps to the pair $(P,Q)$ under the RSK correspondence, then its transpose $M^\top$ maps to the pair $(Q,P)$.
    Consequently, a symmetric matrix corresponds to a pair of identical Gelfand--Tsetlin patterns.
\end{theorem}

For $1\leq k\leq a\wedge b$, let $\hat\Pi_{a,b}^k$ be the set of $k$-tuples of non-intersecting directed paths starting at the points $(1,1),(2,1), \dots, (k,1)$ and ending at the points $(a-k+1,b),(a-k+2,b),\dots,(a,b)$, respectively. 
Then, the theorem follows directly from the next lemma, Lemma~\ref{lemma: k-point paths equivalence}. 
More concretely, if $(P',Q')$ is the image of $M^\top$ under the RSK, then using the lemma we get
\begin{gather*}
    q'_{r1} + q'_{r2} + \dots + q'_{rk} = 
    \max_{(\pi_1,\pi_2,\dots,\pi_k)\in \Pi_{r,n}^k} \, 
    \sum_{(i,j) \in \pi_1 \cup \pi_2 \cup \dots \cup \pi_k}
    m_{ji} 
    = 
    \max_{(\pi_1,\pi_2,\dots,\pi_k)\in \hat\Pi_{r,n}^k} \, 
    \sum_{(i,j) \in \pi_1 \cup \pi_2 \cup \dots \cup \pi_k}
    m_{ji} \\
    = \max_{(\pi_1,\pi_2,\dots,\pi_k)\in \Pi_{n,r}^k} \, 
    \sum_{(i,j) \in \pi_1 \cup \pi_2 \cup \dots \cup \pi_k}
    m_{ij} 
    = 
    p_{r1} + p_{r2} + \dots + p_{rk}
\end{gather*}
and similarly
\[ 
p'_{r1}+p'_{r2}+\dots + p'_{rk} = q_{r1}+q_{r2}+\dots + q_{rk},
\]
so that $(P',Q') = (Q,P).$
    
\begin{lemma} \label{lemma: k-point paths equivalence}
There is a bijection
\[
\Phi:\Pi_{a,b}^k\longrightarrow\hat\Pi_{a,b}^k
\]
such that if
$\Phi(\pi_1,\dots,\pi_k)=(\pi'_1,\dots,\pi'_k)$, then
\[
\pi_1\cup\cdots\cup\pi_k
=
\pi'_1\cup\cdots\cup\pi'_k,
\]
i.e. the two paths use the same set of lattice points. 
\end{lemma}
\begin{proof}
    Let $(\pi_1,\pi_2,\dots, \pi_k)\in \Pi_{a,b}^k$. 
    Note that for each $1\leq i \leq k$,  $\pi_i$ must start with the points
    $(1,i), (2,i),\dots, (k-i+1,i)$
    and end with the points
    $(a-i+1,b-k+i), (a-i+2,b-k+i), \dots, (a,b-k+i)$
    in order for the paths to be non-intersecting. 
    Now for each $1\leq i \leq k$, set $\pi'_{k-i+1}$ to be $\pi_i$ with its starting points
    $(1,i), (2,i),\dots, (k-i+1,i)$ changed to 
    $(k-i+1,1),(k-i+1,2),\dots, (k-i+1,i)$
    and its ending points
    $(a-i+1,b-k+i), (a-i+2,b-k+i), \dots, (a,b-k+i)$
    changed to
    $(a-i+1,b-k+i), (a-i+1,b-k+i+1), \dots, (a-i+1,b)$. 
    Note that after these modifications, the paths $\pi'_1,\dots \pi'_k$ are still non-intersecting and $(\pi'_1,\dots, \pi'_k)\in \hat\Pi_{a,b}^k. $
    Moreover, if we let 
    $T_1 = \{(i,j):i,j\geq 1 \text{ and }i+j \leq k+1\}$
    and 
    $T_2 = \{(i,j): i\leq a, \ j \leq b, \text{ and }i+j \geq a+b-k+1\}$,
    then
    $(\cup_i \pi_i) \cap T_1 = T_1 =  (\cup_i \pi'_i)\cap T_1,$
    $(\cup_i \pi_i) \cap T_2 = T_2 =  (\cup_i \pi'_i)\cap T_2,$
    and 
    $(\cup_i \pi_i) \setminus (T_1 \cup T_2) =  (\cup_i \pi'_i)\setminus (T_1 \cup T_2)$
    since the modifications only occur inside the regions $T_1$ and $T_2$. 
    Hence $\cup_i \pi_i = \cup_i \pi'_i$. 
    The reverse modifications recover the original ensemble uniquely.
    Thus the construction is a bijection, and it preserves the union of
the occupied lattice points.
\end{proof}

For $\lambda = (\lambda_1,\lambda_2,\dots)$ a partition, let
\[
O(\lambda) \defeq \sum_i (-1)^{i-1}\lambda_i
= \lambda_1-\lambda_2+\lambda_3-\dots
\]
be its alternating sum,
which is also the number of odd columns in the Young diagram of $\lambda$. 
The next result tells us that the diagonal sum of a symmetric matrix can be determined from its RSK image using only its shape. 

\begin{proposition}
[Sch\"utzenberger's lemma]
\label{schutzenberger}
    If $M = (m_{ij}) \in (\Z_{\geq 0})^{n \times n}$ is symmetric and $(R,R)$ is its image under the RSK, then
    \[ \sum_{i=1}^n m_{ii} = O(\sh(R)). \]
\end{proposition}

The result was first proved in \cite{schutzenberger} for permutation matrices. We will prove this more general case directly from the classical row insertion definition of the RSK correspondence in Appendix~\ref{appendix}. 

\subsection{RSK push-forward for full-space geometric LPP}\label{subsec: geom lpp}

In this and the next subsection, we take the input to the RSK to be a random matrix 
whose entries are the weights of a geometric LPP environment in $\Z^2$. 
Our goal is to study the corresponding push-forward law of the RSK output, 
since it encodes the joint law of the associated multi-point last passage values and provides a convenient way to compare the full-space and half-space models. 
In this subsection we work in the full-space setting, where the matrix entries are independent; in the next subsection we pass to the half-space setting by imposing reflection symmetry across the diagonal on the matrix.

Let $\alpha,\beta\in (\R_{>0})^n $ be such that $\alpha_i\beta_j \in (0,1)$ for all $i,j$.
Let $M = (m_{ij})$ be an $n\times n$ random matrix with independent entries $m_{ij}$ such that, for each $1\leq i,j\leq n$,
\[
\mathbb{P}(m_{ij} = k)=(1-\alpha_i\beta_j) (\alpha_i \beta_j)^k, \qquad k\geq 0.
\]

Equivalently, $M$ has probability mass function
\[ \mathbb P_{\alpha,\beta}(M=(m_{ij})) 
= \Bigl(\prod_{i,j}(1-\alpha_i\beta_j)\Bigr)\prod_{i,j}(\alpha_i\beta_j)^{m_{ij}}
= Z_{\alpha,\beta} \prod_{r=1}^n \alpha_r^{\sum_i m_{ri}} \beta_r^{\sum_i m_{ir}},
\]
where $Z_{\alpha,\beta} = \prod_{i,j} (1-\alpha_i \beta_j)$. 

Since the RSK map is a bijection, we may use \eqref{eq: row sum} and \eqref{eq: col sum}
to express the row and column sums of $M$ in terms of the RSK output $(P,Q)$, thereby obtaining the push-forward mass function on pairs of Gelfand--Tsetlin patterns:
\[ \mathbb P^{\mathrm{RSK}}_{\alpha,\beta}(P=p,Q=q) 
= Z_{\alpha,\beta} 
\prod_{r=1}^n 
\alpha_r^{\sum_{i=1}^r q_{ri} - \sum_{i=1}^{r-1} q_{r-1,i}}
\beta_r^{\sum_{i=1}^r p_{ri} - \sum_{i=1}^{r-1} p_{r-1,i}}
\mathds{1}_{\{\sh(p) = \sh(q)\}}.
\]
We introduce the following shorthand notation that for $G$ a Gelfand--Tsetlin pattern and $\theta \in \R^n$,
\[ 
\theta^G \defeq  \prod_{r=1}^n  \theta_r^{ \sum_{i=1}^r g_{ri} - \sum_{i=1}^{r-1} g_{r-1,i}}.
\]
(Equivalently, if $G$ is viewed as a semistandard Young tableau, then the exponent of $\theta_r$ is the number of entries equal to $r$.)
Then, the push-forward mass can be expressed as
\begin{equation}
\label{eq: rsk pushforward mass lpp}
\mathbb P^{\mathrm{RSK}}_{\alpha,\beta}(P=p,Q=q)= Z_{\alpha,\beta}\, \alpha^q \beta^p \mathds{1}_{\{\sh(p) = \sh(q)\}}.
\end{equation}

\begin{proposition}
    The marginal distribution of the common shape of $P$ and $Q$ is given by 
    \[
    \mathbb P^{\mathrm{RSK}}_{\alpha,\beta}(\sh(P)=\sh(Q)=\lambda)=Z_{\alpha,\beta} \, s_\lambda(\alpha)s_\lambda(\beta)
    \]
    where $s_\lambda$ is the Schur polynomial. In particular, it is nonzero iff $\lambda$ has at most $n$ parts, and it is invariant under permutations of $\alpha_1,\dots, \alpha_n$ or permutations of $\beta_1,\dots, \beta_n$.
\end{proposition}
\begin{proof}
    We can compute from \eqref{eq: rsk pushforward mass lpp} that
    \begin{align*}
    \mathbb P^{\mathrm{RSK}}_{\alpha,\beta}(\sh(P)=\sh(Q)=\lambda)
    &=
    \sum_{p\in \GT_n(\lambda)} \sum_{q\in \GT_n(\lambda)} Z_{\alpha,\beta} \,\alpha^q \beta^p\\  
    &= Z_{\alpha,\beta}\Bigl( \sum_{q\in \GT_n(\lambda)}\alpha^q\Bigr) \Bigl( \sum_{p\in \GT_n(\lambda)}\beta^p \Bigr) \\
    &= Z_{\alpha,\beta} \, s_\lambda(\alpha)s_\lambda(\beta).
    \end{align*}
    In particular, $Z_{\alpha,\beta}$ is always positive and does not depend on the ordering of the entries of $\alpha$ or $\beta$. The Schur polynomial $s_\lambda$ is symmetric, hence also invariant under permutations of entries of $\alpha$ or $\beta$. Finally, $s_\lambda(\alpha)$ or $s_\lambda(\beta)$ is zero iff $\GT_n(\lambda)=\emptyset$ iff $\lambda$ has more than $n$ parts. 
\end{proof}

\begin{proposition}
 \label{prop: rsk Q given lambda}
    Under this push-forward measure, for every partition $\lambda$ of at most $n$ parts, the conditional distribution of $Q$ given $\sh(Q)=\lambda$
    is
    \[
    \mathbb P^{\mathrm{RSK}}_{\alpha,\beta} (Q=q \,|\, \sh(Q)=\lambda) 
    = \frac{1}{s_\lambda(\alpha)}\, \alpha^q, \qquad q\in \GT_n(\lambda).
    \]
\end{proposition}
\begin{proof}
    The conditional joint distribution of $P$ and $Q$ given their common shape being $\lambda$ is
    \[
    \mathbb P^{\mathrm{RSK}}_{\alpha,\beta}(P=p,Q=q \, | \, \sh(P)=\sh(Q)=\lambda) = \frac{Z_{\alpha,\beta}\,\alpha^q \beta^p}{Z_{\alpha,\beta}\, s_\lambda(\alpha)s_\lambda(\beta)} 
    =\frac{1}{s_\lambda(\alpha)s_\lambda(\beta)} \alpha^q \beta^p
    \]
    for 
    $p,q \in \GT_n(\lambda)$.
    The result follows by summing over $p\in \GT_n(\lambda)$.
\end{proof}

\subsection{RSK push-forward for half-space geometric LPP}\label{subsec: half space geom lpp}

Let $\gamma = (\gamma_1,\dots, \gamma_n) \in (\R_{>0})^n$ be a sequence of parameters such that $\gamma_i\gamma_j \in (0,1)$ for all $1\leq i<j \leq n$. 
Let $c \in \R_{>0}$ be an additional parameter such that $c \gamma_i \in (0,1)$ for all $i$. 
Let $\tilde{M}=(\tilde{m}_{ij})$ be a symmetric $n\times n$ random matrix, such that the upper-triangular entries $\{\tilde m_{ij}:1\leq i\leq j \leq n\}$ are independently distributed with 
\[
\mathbb{P}(\tilde{m}_{ij} =k) = (1-\gamma_i\gamma_j)(\gamma_i\gamma_j)^k, \qquad k\geq0,
\]
for $i<j$
and
\[
\mathbb{P}(\tilde{m}_{ii}= k)=(1-c\gamma_i)(c\gamma_i)^k, \qquad k\geq 0.
\]

Equivalently, the matrix $\tilde{M} \in (\Z_{\geq 0})^{n\times n}_{\mathrm{sym}}$ has probability mass function
\[ \tilde{\mathbb P}_{c,\gamma}(\tilde{M}) 
= 
\prod_{i<j} (1-\gamma_i\gamma_j)(\gamma_i\gamma_j)^{\tilde{m}_{ij}} 
\prod_{i=1}^n (1-c\gamma_i)(c\gamma_i)^{\tilde{m}_{ii}}
= \tilde{Z}_{c,\gamma} \,
c^{\sum_{i=1}^n \tilde{m}_{ii}}
\prod_{r=1}^n \gamma_r^{\sum_{i=1}^n \tilde{m}_{ir}},
\]
where 
\[
\tilde{Z}_{c,\gamma} = \prod_{i<j} (1-\gamma_i\gamma_j)\prod_{i=1}^n (1-c\gamma_i).
\]

By Theorem~\ref{RSK symmetry}, the RSK restricts to a bijection between symmetric matrices and pairs of identical Gelfand--Tsetlin patterns. 
Let $(R,R)$ be the image of $\tilde{M}$ under the RSK.
Then, using \eqref{eq: row sum} and Proposition~\ref{schutzenberger} we can again write the push-forward law on $R$ as
\begin{equation} \label{pmf half space lpp}
    \tilde{\mathbb P}^{\mathrm{RSK}}_{c,\gamma}(R=r) = \tilde{Z}_{c,\gamma}\, c^{O(\sh(r))} \gamma^r
\end{equation} 
for $r \in \GT_n$. 

\begin{proposition}
    \label{lpp matrix change of parameter}
    The marginal distribution of $\sh(R)$ is given by 
    \[
    \tilde{\mathbb{P}}^{\mathrm{RSK}}_{c,\gamma}(\sh(R) = \lambda) = \tilde{Z}_{c,\gamma}  s_\lambda(\gamma) c^{O(\lambda)}.
    \]
    In particular, it is nonzero iff $\lambda$ has at most $n$ parts, and it does not depend on the order of the parameters $\gamma_1,\dots,\gamma_n$. 
\end{proposition}
\begin{proof}
    Sum \eqref{pmf half space lpp} over all $r\in \GT_n(\lambda)$ and use the identity $s_\lambda(\gamma)=\sum_{G\in \GT_n(\lambda)}\gamma^G$ to obtain the stated expression. The final claim follows because $c^{O(\lambda)}$ and  $\tilde{Z}_{c,\gamma}$ are both nonzero and invariant under permutations of the entries of $\gamma$, together with standard symmetry and non-vanishing properties of $s_\lambda$.
\end{proof}

\begin{proposition}
\label{prop: RSK R given lambda}
    For $\lambda$ a partition of at most $n$ parts, the conditional distribution of $R$ given $\sh(R)=\lambda$ is
    \[ \tilde{\mathbb{P}}^{\mathrm{RSK}}_{c,\gamma}(R=r\, | \, \sh(R)=\lambda) 
    = \frac{1}{s_\lambda(\gamma)} \, 
      \gamma^r
    \]
    for $r\in \GT_n(\lambda)$.
\end{proposition}
\begin{proof}
    From \eqref{pmf half space lpp} and Proposition~\ref{lpp matrix change of parameter}, we immediately have that for $r \in \GT_n(\lambda)$,
    \[
    \tilde{\mathbb{P}}^{\mathrm{RSK}}_{c,\gamma}(R =r\, | \, \sh(R)=\lambda) 
    = 
    \frac{\tilde{Z}_{c,\gamma}\, c^{O(\lambda)} \gamma^r}
    {\tilde{Z}_{c,\gamma}  s_\lambda(\gamma) c^{O(\lambda)}}
    = \frac{1}{s_\lambda(\gamma)}\, 
      \gamma^r.
    \]
\end{proof}

By comparing the conditional laws in Proposition~\ref{prop: rsk Q given lambda} and Proposition~\ref{prop: RSK R given lambda}, we immediately get the following result. 

\begin{theorem}
    \label{bridge LPP}
    Let $M$ be a random matrix corresponding to the full-space geometric LPP environment with parameters $(\alpha,\beta)$, with RSK image being $(P,Q)$. 
    Let $\tilde M$ be a random symmetric matrix corresponding to the half-space geometric LPP environment with parameters $(c,\gamma)$, with RSK image $(R,R)$.
    Assume in addition that $\gamma=\alpha$. Then, 
    \[
    \mathcal L(R\mid \sh(R))
    =
    \mathcal L(Q\mid \sh(Q)).
    \]
\end{theorem}

\subsection{Geometric RSK correspondence}
\label{subsec: gRSK}

The geometric RSK (gRSK) provides a bijective correspondence from the set of positive matrices $(\R_{>0})^{n\times n}$ to
itself. 
It can be defined via row insertion as in \cite[Section 2.1]{corwinTropicalCombinatoricsWhittaker2014a}, via local moves as in \cite{oconnell}, or via discrete Pitman transforms as in \cite{corwin}.
We will adopt the description using non-intersecting paths (see \cite[Section 2.2]{corwinTropicalCombinatoricsWhittaker2014a} and \cite{noumiTropicalRobinsonSchenstedKnuthCorrespondence2004}),
because it makes the most direct connection to the setting of directed polymers. 

Recall the notion of directed paths and the notation $\Pi_{a,b}^k$ that we introduced in Subsection~\ref{subsec: RSK}.

Let $W = (w_{ij}) \in (\R_{>0})^{n\times n}$. 
The gRSK associates to $W$ the output matrix $T = (t_{ij}) \in (\R_{>0})^{n\times n}$, defined implicitly by the following identities.
The lower-triangular entries of $T$ are characterized by that, for every $1\leq k \leq r \leq n$,
\begin{equation}
\label{eq: gRSK def P}
    \prod_{i=1}^k t_{n+1-i,r+1-i}
    = \sum_{(\pi_1,\pi_2,\dots,\pi_k) \in \Pi_{n,r}^k} \prod_{(i,j)\in \pi_1\cup \dots\cup \pi_k} w_{ij}.
\end{equation}
The upper-triangular entries are characterized by that,
for every $1 \leq k \leq r < n$,
\begin{equation}
\label{eq: gRSK def Q}
    \prod_{i=1}^k t_{r+1-i,n+1-i}
    =
    \sum_{(\pi_1,\pi_2,\dots,\pi_k) \in \Pi_{r,n}^k} \prod_{(i,j)\in \pi_1\cup \dots\cup \pi_k} w_{ij}. 
\end{equation}

The symmetry property of the gRSK follows from the definition and Lemma~\ref{lemma: k-point paths equivalence}, as in the case of the classical RSK:

\begin{theorem}
    \label{thm: gRSK symmetry}
    If $\mathrm{gRSK}(W)=T$, then $\mathrm{gRSK}(W^\top)=T^\top$. In particular, if $W$ is symmetric, then $T$ is symmetric.
\end{theorem}

Also as in the classical RSK, the gRSK output encodes products of input entries along rows and columns. 
For each $1\leq r\leq n$, the defining relations imply
\begin{equation}
\label{eq: row product}
    \prod_{i=1}^n w_{ri} = \frac{ t_{r,n}\, t_{r-1,n-1} \cdots t_{1,n-r+1}}{t_{r-1,n}\, t_{r-2,n-1} \cdots t_{1,n-r+2}}
\end{equation}
and
\begin{equation}
\label{eq: col product}
    \prod_{i=1}^n w_{ir} = \frac{t_{n,r}\, t_{n-1,r-1} \cdots t_{n-r+1,1}}{t_{n,r-1}\, t_{n-1,r-2} \cdots t_{n-r+2,1}},
\end{equation}
with the convention that when $r=1$ the denominators are equal to 1. 
These identities will later be used to identify the push-forward laws of polymer environments under gRSK.

A crucial analytic property of gRSK is that it preserves Lebesgue measure in logarithmic coordinates.
\begin{lemma}
[{\cite[Theorem~3.1]{oconnell}}]
\label{volume-preserving}
    The gRSK mapping in logarithmic variables
    \[ (\log w_{ij})_{1\leq i,j\leq n} \ \longmapsto \  (\log t_{ij})_{1\leq i,j \leq n} \]
    has Jacobian determinant $\pm 1$. 
\end{lemma}

To mirror the $(P,Q)$ viewpoint of the classical RSK, it is convenient to write $T$ as a pair of triangular patterns, which are also known as geometric Young tableaux.

A \emph{triangular pattern} with $n$ rows is an array $Y=(y_{ij})_{1\leq j\leq i\leq n}$ of positive reals, usually written in the form
\[
Y \;=\;
\begin{array}{ccccccc}
  & &   & y_{11} &   & &   \\[6pt]
  & & y_{22} &   & y_{21} & &   \\[6pt]
& \iddots &   & \vdots &   & \ddots &  \\[6pt]
y_{nn} & & &  \cdots & & &  y_{n1}.
\end{array}
\]
We define its \emph{shape} by
\[
\sh(Y) \defeq (y_{n1},\dots,y_{nn}) \in (\R_{>0})^n.
\]
Note that this is its last row read from right to left. We will use $\mathrm{Tri}_n$ to denote the set of triangular patterns with $n$ rows and $\mathrm{Tri}_n(\lambda)$ for those having the shape $\lambda \in (\R_{>0})^n.$

Given the square matrix $T=(t_{ij})$, define two triangular patterns $P=(p_{ij})$ and $Q=(q_{ij})$ by splitting $T$ along its diagonal. 
Specifically, $P$ records the lower-triangular part of $T$ and $Q$ records the upper-triangular part,  with the indexing convention
\[
p_{ij}\defeq t_{n+1-j,\,i-j+1}, \qquad q_{ij}\defeq t_{i-j+1,\, n+1-j}.
\]
We will conveniently identify $T$ with this pair $(P,Q)=(P(T),Q(T))$.
Written in triangular form, we have
\[
P \;=\;
\begin{array}{ccccccc}
  & &   & p_{11} &   & &   \\[6pt]
  & & p_{22} &   & p_{21} & &   \\[6pt]
& \iddots &   & \vdots &   & \ddots &  \\[6pt]
p_{nn} & & &  \cdots & & &  p_{n1}
\end{array}
\ = \ 
\begin{array}{ccccccc}
  & &   & t_{n1} &   & &   \\[6pt]
  & & t_{n-1,1} &   & \ \ t_{n2} & &   \\[6pt]
& \iddots &   & \vdots &   & \ddots &  \\[6pt]
t_{11}\  & & &  \cdots & & & \ t_{nn} 
\end{array}
\]
and
\[
Q \;=\;
\begin{array}{ccccccc}
  & &   & q_{11} &   & &   \\[6pt]
  & & q_{22} &   & q_{21} & &   \\[6pt]
& \iddots &   & \vdots &   & \ddots &  \\[6pt]
q_{nn} & & &  \cdots & & &  q_{n1}
\end{array}
\ = \ 
\begin{array}{ccccccc}
  & &   & t_{1n} &   & &   \\[6pt]
  & & t_{1,n-1} &   & t_{2n} & &   \\[6pt]
& \iddots &   & \vdots &   & \ddots &  \\[6pt]
t_{11} & & &  \cdots & & &  t_{nn}. 
\end{array}
\]

With this convention, the reversed diagonal
\[
\mathrm{revdiag}(T) \defeq (t_{nn},t_{n-1,n-1},\dots, t_{11})
\]
is exactly the common shape of $P$ and $Q$.
Thus we may equivalently view gRSK as a bijection
\[
(\R_{>0})^{n\times n} \longrightarrow \{(P,Q):\sh(P)=\sh(Q)\}.
\]

Also, transposing $T$ results in interchanging $P$ and $Q$, so gRSK restricts to a bijection between symmetric matrices and identical pairs of triangular patterns. 

Rewriting \eqref{eq: row product}--\eqref{eq: col product} in terms of $P,Q$ gives, for each $1\leq r\leq n$,
\begin{equation}
    \label{eq: row product in Q}
    \prod_{i=1}^n w_{ri} = \frac{q_{r1} q_{r2} \dots q_{rr}}{q_{r-1,1} q_{r-1,2} \dots q_{r-1,r-1}}
\end{equation}
and 
\begin{equation}
    \label{eq: col product in P}
    \prod_{i=1}^n w_{ir} = \frac{p_{r1} p_{r2} \dots p_{rr}}{p_{r-1,1} p_{r-1,2} \dots p_{r-1,r-1}},
\end{equation}
again with the convention that the denominators equal 1 when $r=1.$

We record two combinatorial identities of the gRSK that will be useful later for computing push-forward laws.

Following \cite[(2.10)]{oconnell}, with $s=1/2$, 
define for a triangular pattern $Y=(y_{ij})$ with $n$ rows that 
\[ \mathcal{F}(Y) 
\defeq 
\frac{1}{2y_{nn}} + \sum_{1\leq j\leq i< n} \frac{y_{i+1,j+1}}{y_{ij}} 
+ \sum_{1\leq j < i \leq n} \frac{y_{i-1,j}}{y_{ij}}.\]
Also, for $T$ identified with $(P,Q)$, define 
\[ \mathcal{E}(T) \defeq \frac{1}{t_{11}} + \sum_{i,j} \frac{t_{i-1,j} + t_{i,j-1}}{t_{ij}}= \mathcal{F}(P)+\mathcal{F}(Q). \] 
In this sum, terms with an index outside $\{1,\dots, n\}^2$ are interpreted as zero. 
This definition matches with $\mathcal E_1(T)$ in \cite{oconnell}.

\begin{lemma}
[{\cite[Theorem~3.2]{oconnell}}]
\label{inverse-sum}
    Under gRSK,
    \[ 
    \sum_{i,j} \frac{1}{w_{ij}} = \mathcal{E}(T) = \mathcal{F}(P)+\mathcal{F}(Q).
    \]
\end{lemma}

Next, for $x = (x_1,x_2,\dots, x_n) \in (\R_{>0})^n$, define
\[ 
\mathcal{D}(x) 
 = \prod_{i=1}^n x_i^{(-1)^{i+1}}
= \frac{x_1x_3x_5\dots}{x_2x_4x_6\dots}.
\] 

\begin{lemma}
[{\cite[Lemma~5.1(b)]{oconnell}}]
\label{diagonal product}
    
    If $W$ is symmetric, 
    then
    \[ 
    4^{\lfloor n/2 \rfloor}
    \prod_{i=1}^n 
    w_{ii} 
    = \prod_{i=1}^n t_{i,i}^{(-1)^{n-i}}
    =
    \mathcal{D}(\mathrm{revdiag}(T))
    = \mathcal{D}(\sh(P))
    = \mathcal{D}(\sh(Q))
    .
    \]
\end{lemma}

\subsection{gRSK push-forward for full-space log-gamma polymer}\label{subsec: log-gamma polymer}

In this and the next subsection, we apply the gRSK to matrices coming from log-gamma polymer environments and study the push-forward laws on the output triangular patterns. 

Fix parameters $\alpha,\beta \in \R^n$ such that $\alpha_i+\beta_j > 0 $ for all $1\leq i,j \leq n$. 
Under the log-gamma polymer model, we take the input matrix $W = (w_{ij})_{1\leq i,j\leq n}$ to have independent entries with 
\[
w_{ij} \sim \mathrm{Gamma}^{-1}(\alpha_i+\beta_j),
\]
i.e. with density $\Gamma(\alpha_i+\beta_j)^{-1}x^{-(\alpha_i+\beta_j)-1}e^{-1/x}$ on $\R_{> 0}.$

Equivalently, the joint law of $W$ as a probability measure on $(\R_{> 0})^{n\times n}$ is given by

\[
    \nu_{\alpha,\beta} (dw) = \frac{1}{\mathcal Z_{\alpha,\beta}} \prod_{i,j} w_{ij}^{-\alpha_i-\beta_j} e^{-\sum_{i,j} 1 / w_{ij}} \prod_{i,j} \frac{dw_{ij}}{w_{ij}}
    ,
\]
where 
$ \mathcal Z_{\alpha,\beta} = \prod_{i,j} \Gamma(\alpha_i+\beta_j)$. 

We introduce a shorthand notation that for $Y \in \mathrm{Tri}_n$ and $\theta \in \R^n$,
\[
Y^\theta \defeq \prod_{k=1}^n \left(
\frac{y_{k,k} \, y_{k,k-1} \dots y_{k,1}}{y_{k-1,k-1}\, y_{k-1,k-2} \dots y_{k-1,1}} \right)
^{\theta_k}
\]
with the convention that when $k=1$ the denominator equals 1. 

Then, using the identities \eqref{eq: row product in Q} and \eqref{eq: col product in P}, we can write 
\begin{equation} 
\label{eq: W in terms of P,Q}
    \prod_{i,j} w_{ij}^{-\alpha_i-\beta_j} = \prod_{i=1}^n \left( \prod_{j=1}^n w_{ij} \right)^{-\alpha_i} \prod_{j=1}^n \left( \prod_{i=1}^n w_{ij}\right)^{-\beta_j} = P^{-\beta} Q^{-\alpha}.
\end{equation}

Together with Lemmas~\ref{volume-preserving} and \ref{inverse-sum}, it follows that the 
push-forward of $\nu_{\alpha,\beta}$ under gRSK is 
\[
    (\nu_{\alpha,\beta} \circ \mathrm{gRSK}^{-1})(dt) = \frac{1}{\mathcal Z_{\alpha,\beta}} P(t)^{-\beta} Q(t)^{-\alpha} e^{-\mathcal{E}(t)} \prod_{i,j} \frac{dt_{ij}}{t_{ij}}
\]
(This is also \cite[Corollary~3.3]{oconnell}).

As $\mathcal{E}(T) = \mathcal{F}(P)+\mathcal{F}(Q)$, 
the preceding density cleanly factorizes if we fix the diagonal of $T$, and then integrating out the $P$-coordinates gives us the conditional law of $Q$ given its shape.
\begin{proposition}
\label{prop: gRSK Q given shape}
    The regular conditional distribution kernel of $Q$ given $\sh(Q)$ is
    \[
    K_{Q \mid \sh(Q)}(\mu,dq) = 
    \frac{1}{\mathcal Z_{\alpha}(\mu)} q^{-\alpha} 
    e^{-\mathcal{F}(q)}
    \prod_{1\leq j\leq i < n} \frac{dq_{ij}}{q_{ij}}, \qquad q\in \mathrm{Tri}_n(\mu),
    \]
    where
    $\mu \in (\R_{>0})^n$
    and
    \[
    \mathcal Z_\alpha(\mu) = \int_{\mathrm{Tri}_n(\mu)} q^{-\alpha} 
    e^{-\mathcal{F}(q)}
    \prod_{1\leq j\leq i < n} dq_{ij}/q_{ij}.
    \]
\end{proposition}

\subsection{gRSK push-forward for half-space log-gamma polymer}\label{subsec: half space log gamma}

Now we turn to the half-space or symmetric environment.

Let $\gamma = (\gamma_1,\dots, \gamma_n) \in \R^n$ be a sequence of parameters such that $\gamma_i+\gamma_j>0$ for all $1 \leq i<j \leq n$.  
Let $c \in \R$ be an additional parameter such that $c + \gamma_i > 0$ for all $i$. 
Let $\tilde{W}=(\tilde{w}_{ij})$ be a symmetric $n\times n$ random matrix, such that the upper-triangular entries $\{\tilde w_{ij}:1\leq i\leq j \leq n\}$ are independently distributed with 
\[
\tilde w_{ij} \sim \mathrm{Gamma}^{-1}(\gamma_i+\gamma_j)
\]
for $i<j$
and
\[
2 \tilde w_{ii} \sim \mathrm{Gamma}^{-1}(c+\gamma_i).
\]

Equivalently, this is defined by the probability measure on the space $(\R_{>0})^{n\times n}_{\mathrm{sym}}$, identified with that of the upper-triangular entries, by
\[ 
\tilde{\nu}_{c,\gamma} (d\tilde w) = 
\frac{1}{\tilde{\mathcal Z}_{c,\gamma}}
\prod_{i < j} \tilde{w}_{ij}^{-\gamma_i-\gamma_j} 
\prod_{i} \tilde{w}_{ii}^{-c-\gamma_i} 
\exp\left(-\sum_{i<j}\frac{1}{\tilde{w}_{ij}}-\sum_{i}\frac{1}{2\tilde{w}_{ii}} \right) 
\prod_{i\leq j} 
\frac{d\tilde{w}_{ij}}{\tilde{w}_{ij}}
\]
where 
\[
\tilde{\mathcal Z}_{c,\gamma} = 2^{\sum_i (c+ \gamma_i)}\prod_i \Gamma(c+\gamma_i)\prod_{i<j} \Gamma(\gamma_i+\gamma_j).
\] 

The gRSK restricts to a bijection $\tilde W \mapsto (R,R)$ between symmetric matrices with positive entries and pairs of identical triangular patterns. 
The volume-preserving property still holds in this case:
\begin{lemma}
[{\cite[Theorem~5.2]{oconnell}}]
\label{volume-preserving symmetric}
    The gRSK on symmetric matrices in logarithmic variables
    \[
    (\log \tilde{w}_{ij})_{1\leq i \leq j \leq n} \longmapsto (\log r_{ij})_{1 \leq j\leq i \leq n}
    \]
    has Jacobian determinant $\pm 1$. 
\end{lemma}

\begin{proposition}
\label{prop: gRSK pushforward law of R}
    The push-forward law of $R\in \mathrm{Tri}_n$ is given by
    \[
    (\tilde\nu_{c,\gamma} \circ \mathrm{gRSK}^{-1}) (dr) =
    \frac{4^{\lfloor n/2 \rfloor c}}{\tilde{\mathcal Z}_{c,\gamma}}
    r^{-\gamma} e^{-\mathcal{F}(r)} 
    \mathcal{D}(\sh(r))^{-c}
    \prod_{j \leq i} \frac{dr_{ij}}{r_{ij}}.
    \]
\end{proposition}
\begin{proof}
    From \eqref{eq: row product in Q} we have
    \[ \prod_{i<j} \tilde{w}_{ij}^{-\gamma_i-\gamma_j}
    \prod_{i} \tilde{w}_{ii}^{-\gamma_i} = \prod_{i} \left( \prod_{j} \tilde{w}_{ij} \right)^{-\gamma_i} = R^{-\gamma}.\]
    From Lemma~\ref{inverse-sum} we have that
    \[ \exp\left(-\sum_{i<j}\frac{1}{\tilde{w}_{ij}}-\sum_{i}\frac{1}{2\tilde{w}_{ii}} \right) 
     = \exp \left(-\frac12 \sum_{i,j} \frac{1}{\tilde{w}_{ij}}\right) = e^{-\mathcal{F}(R)}.\]
    From Lemma~\ref{diagonal product} we have
    \[ \prod_i \tilde{w}_{ii}^{-c} = (4^{-\lfloor n/2 \rfloor} \mathcal{D}(\sh(R)))^{-c}.\]
    The result follows from these together with Lemma~\ref{volume-preserving symmetric}.
\end{proof}

\begin{proposition}
\label{polymer matrix change of parameter}
    The marginal distribution of $\sh(R)$ is invariant under the permutation of the parameters $\gamma_1,\dots,\gamma_n$.
\end{proposition}
\begin{proof}
    This follows because the marginal law of $\sh(R)$ can be written explicitly in terms of the $\mathrm{GL}(n,\R)$-Whittaker function $\Psi_\gamma^n$, which is a well-known symmetric function, as
    \[
    \frac{1}{\tilde{\mathcal Z}^{\mathrm{sh}}_{c,\gamma}} 
    \mathcal{D}((r_{n1},\dots, r_{nn}))^{-c}
    e^{-1/2r_{nn}} \Psi^n_\gamma(r_{n1},\dots, r_{nn}) \prod_{i=1}^n \frac{dr_{ni}}{r_{ni}},   
    \]
    where 
    \[
    \tilde{\mathcal Z}^{\mathrm{sh}}_{c,\gamma} = \int_{(\R_{>0})^n} \mathcal{D}(x)^{-c}
    e^{-1 / 2x_{n}} \Psi^n_\gamma(x) \prod_i dx_{i}/x_i. 
    \]
    See \cite[Corollary~5.3]{oconnell}, the remark following (2.11) therein, and \cite{gerasimovBaxterOperatorArchimedean2008, kharchevIntegralRepresentationsEigenfunctions2001}.
    Alternatively, it can be shown using the more general framework of Macdonald processes, as has been essentially proved in the proof of \cite[Lemma~2.7]{barraquandStationaryMeasuresLoggamma2023}. 
\end{proof}

By fixing $\sh(R)$ in the density given in Proposition~\ref{prop: gRSK pushforward law of R}, we get
the following.
\begin{proposition}
\label{prop: gRSK R given shape}
    The regular conditional distribution kernel of $R$ given $\sh(R)$ is
    \[
    K_{R \mid \sh(R)}(\mu,dr)= 
    \frac{1}{{\mathcal Z}_{\gamma}(\mu)} 
    r^{-\gamma} 
    e^{-\mathcal{F}(r)}
    \prod_{1\leq j \leq i<n} \frac{dr_{ij}}{r_{ij}}, 
    \qquad r\in \mathrm{Tri}_n(\mu)
    \]
    where $\mu \in (\R_{>0})^n$ and
    \[
    {\mathcal Z}_{\gamma}(\mu)
    = \int_{\mathrm{Tri}_n(\mu)} r^{-\gamma} 
    e^{-\mathcal{F}(r)}
    \prod_{1\leq j \leq i<n} \frac{dr_{ij}}{r_{ij}}.
    \]
\end{proposition}

By comparing the conditional laws in Propositions~\ref{prop: gRSK Q given shape} and \ref{prop: gRSK R given shape}, we get the following result, analogous to what we have in the LPP settings. 

\begin{theorem}
\label{bridge polymer}
    Let $W$ be a random matrix corresponding to the full-space log-gamma polymer environment with parameters $(\alpha,\beta)$, with gRSK image $(P,Q)$. 
    Let $\tilde W$ be a random symmetric matrix corresponding to the half-space log-gamma polymer environment with parameters $(c,\gamma)$, 
    with gRSK image $(R,R)$.
    Assume in addition that $\gamma=\alpha$. 
    Then, 
    \[
    \mathcal L(R\mid \sh(R))
    =
    \mathcal L(Q\mid \sh(Q)).
    \]
\end{theorem}

\section{Invariance results for half-space models} \label{main section}

In this section, we prove our main invariance theorems using their full-space counterparts and distributional identities derived in the previous section. 

\subsection{Settings and notation}

We retain the notation from Sections~\ref{sec: intro} and \ref{sec: preliminary}, and adopt some further notation from \cite{dauvergne} for this section. 

For the general setting, let $M = \{M_x : x\in \Z^2\}$ be a two-dimensional array of real-valued random variables.

We say that $Z_M$ is defined as the last passage value on $M$ if we let
\[ 
Z_M(u) = \max_\pi \sum_{x\in \pi} M_x
\]
for any $u \in \zbox$, where the max is over all $u$-paths $\pi$,
and
\[
Z_M(\mathbf{u}) = \max_\pi \sum_{i=1}^k \sum_{x\in \pi_i} M_x,
\]
for any $\mathbf{u} = (u_1,\dots,u_k) \in (\zbox)^k$,
where the max is taken over all $k$-tuples of disjoint paths $\pi = (\pi_1,\dots, \pi_k)$, where each $\pi_i$ is a $u_i$-path. 

Since we will be dealing with LPP and directed polymers at the same time in this section, we will use the same notation for the directed polymer model. 
We say that $Z_M$ is defined as the partition function on $M$ if we let
\[
Z_M(u) = \sum_\pi \prod_{x\in \pi} M_x
\]
for any $u \in \zbox$,
where the summation is taken over all $u$-paths $\pi$,
and 
\[ Z_M(\mathbf{u}) = \sum_\pi \prod_{i=1}^k \prod_{x\in \pi_i} M_x\]
for $\mathbf{u} = (u_1,\dots,u_k) \in (\zbox)^k$, 
where the summation is taken over all $k$-tuples of disjoint paths $\pi = (\pi_1,\dots, \pi_k)$ with each $\pi_i$ being a $u_i$-path.

By an \emph{environment} or a \emph{model}, we refer to such an array $M$ together with a choice of $Z_M$, being either the last passage value or the partition function. 

In both cases, $Z_M$ is only defined for vectors $\mathbf{u}$ for which a set of disjoint paths exists. We call these vectors \emph{valid}, and use $\mathcal{E}$ to denote the set of valid vectors. 
More generally, for $U \subseteq \zbox$,
we define
\[
\mathcal E(U) \defeq
\{ \mathbf{u}\in \mathcal E:u_i\in U \ \forall i\}
\]
 which is the set of valid vectors whose coordinates all lie in $U$.
We identify $u\in \zbox$ with the one-component valid vector $(u) \in \mathcal E$. 

For $u = (u^-;u^+)\in \zbox$ and $1 \leq k \leq \min\{u^+_1-u^-_1+1,\;u^+_2-u^-_2+1\}$, define
\[ 
u^k \defeq \bigl((u^-;u^+ - (k-1,0)),(u^-+(1,0);u^+-(k-2,0)),\dots, (u^- + (k-1,0);u^+)\bigr)
\in (\zbox)^k,
\]
and
\[
D(u) \defeq \{u^k : 1\leq k \leq \min(u^+_1-u^-_1+1,u^+_2-u^-_2+1)\}.
\]

Using these, we define the extended endpoint set for each $U \subseteq \zbox$,
\[
\mathcal{E}_D(U) \defeq \{ \mathbf{u}= (u_1^{k_1}, \dots, u_m^{k_m}) \in \mathcal{E}:\, u_i \in U, k_i,m\in \Z_{>0}\}. 
\]

For $u = (u^-;u^+)\in \zbox$, define the sets of endpoints
\begin{align*}
    H(u) &\defeq  \cup \{D((u^-,u^+-(0,i))):i\in \{0,1,\dots, u^+_2-u^-_2\}\},  \\
    V(u) &\defeq \cup \{D((u^-,u^+-(i,0))):i\in \{0,1,\dots, u^+_1-u^-_1\}\},   \\
    \bar{H}(u) &\defeq \mathcal{E} ( \{v\in \zbox: v \subseteq u, v^-_1=u^-_1,v^+_1=u^+_1\} ) , \\
    \bar{V}(u) &\defeq \mathcal{E} ( \{v\in \zbox: v \subseteq u, v^-_2=u^-_2,v^+_2=u^+_2\} ) . 
\end{align*}

Also, we will identify $u$ with the box $[u^-_1,u^+_1]\times [u^-_2,u^+_2]$, or the set of lattice points contained in this box, depending on the context. 

To avoid confusion, we will use the symbol $\sqcup$ to denote the unions in the sense of lattice points.
For example, for $U \subseteq \zbox$, 
\[
\sqcup U = \{(x,y)\in \Z^2: \exists u\in U, (x,y) \in u =[u^-_1,u^+_1]\times [u^-_2,u^+_2]\}.
\]

When discussing half-space environments, we write
\[
u^\top \defeq (u^-_2,u^-_1;u^+_2,u^+_1)
\]
for the reflection of  $u = (u^-_1,u^-_2;u^+_1,u^+_2)\in \zbox$ across the diagonal $\{(i,i):i\in \Z\}$.
More generally, for a region $R\subseteq \Z^2$, we write 
\[
R^\top \defeq \{(i,j)\in \Z^2:(j,i)\in R\}.
\]
The regions we will consider will usually be unions of boxes in $\zbox$.

\subsection{Correspondence between half-space and full-space}

In this subsection we convert the RSK and gRSK push-forward identities from Section \ref{sec: preliminary} into a form that can be applied directly to passage values and polymer partition functions. The main point is that, after conditioning on the diagonal data $Z|_{D(u)}$, the collection of observables $Z|_{\bar V(u)}$ has the same conditional law in the half-space environment as in the corresponding full-space environment. This will be the basic mechanism by which the full-space invariance results are transferred to the half-space setting.

We first state this correspondence for the three integrable models considered in this paper.

\begin{theorem}\label{bridge}
    Let $(M,Z_M)$ be a full-space environment and $(\tilde M,Z_{\tilde M})$ the corresponding half-space environment, chosen from one of the following three cases.
    Note that in each case the off-diagonal half-space parameters are chosen to match the horizontal parameter sequence $\alpha$ for the full-space environment. 
    \begin{enumerate}
        \item 
        Fix $\alpha,\beta \in (\R_{>0})^\Z$ and $c\in \R_{>0}$ with $\alpha_i\beta_j \in (0,1)$ for all $i,j \in \Z$, $\alpha_i\alpha_j \in (0,1)$ for all $i \neq j$, and $c\alpha_i \in (0,1)$ for all $i\in \Z$. 
        Let $M$ be an array of independent geometric random variables with $\mathbb{P}(M_{(i,j)} = k) = (1-\alpha_i\beta_j)(\alpha_i\beta_j)^k,$ for $k = 0,1,\dots$.
        Let $\tilde{M}$ be a symmetric array of geometric random variables, so that $\{\tilde{M}_{(i,j)}:i\leq j\}$ are independent with $\mathbb{P}(\tilde{M}_{(i,j)} = k) = (1-\alpha_i\alpha_j)(\alpha_i\alpha_j)^k$ for all $i<j$ and $\mathbb{P}(\tilde{M}_{(i,i)} = k) = (1-c\alpha_i)(c\alpha_i)^k$ for all $i$, and $\tilde{M}_{(i,j)} = \tilde{M}_{(j,i)}$ for all $i,j$. 
        Define $Z_M$ and $Z_{\tilde{M}}$ to be the last passage values on $M$ and $\tilde{M}$. 
        
        \item Fix $\alpha,\beta \in \R^\Z$ and $c \in \R$ with $\alpha_i + \beta_j > 0$ for all $i,j$, $\alpha_i+\alpha_j > 0$ for all $i\neq j$, and $c+\alpha_i >0$ for all $i$.  
        Let $M$ be an array of independent exponential random variables with $\mathbb{E}M_{(i,j)} = (\alpha_i+\beta_j)^{-1}$ for each $i,j$.
        Let $\tilde{M}$ be a symmetric array of exponential random variables, so that $\{\tilde{M}_{(i,j)}:i\leq j\}$ are independent with $\mathbb{E}\tilde{M}_{(i,j)} = (\alpha_i+\alpha_j)^{-1}$ for all $i<j$ and $\mathbb{E}\tilde{M}_{(i,i)} = (c+\alpha_i)^{-1}$ for all $i$, and that $\tilde{M}_{(i,j)} = \tilde{M}_{(j,i)}$ for all $i,j$. 
        Define $Z_M$ and $Z_{\tilde{M}}$ to be the last passage values on $M$ and $\tilde{M}$. 
        
        \item Fix $\alpha,\beta \in \R^\Z$ and $c \in \R$ with $\alpha_i + \beta_j > 0$ for all $i,j$, $\alpha_i+\alpha_j > 0$ for all $i\neq j$, and $c+\alpha_i >0$ for all $i$.  
        Let $M$ be an array of independent inverse-gamma random variables with $M_{(i,j)}$ having parameter $\alpha_i+\beta_j$ (i.e., having density $\frac{1}{\Gamma(\alpha_i+\beta_j)} x^{-\alpha_i-\beta_j-1} e^{-1/x}$) for each $i,j$. 
        Let $\tilde{M}$ be a symmetric array of random variables so that $\{\tilde{M}_{(i,j)}:i\leq j\}$ are independent with $\tilde{M}_{(i,j)}$ having inverse-gamma distribution with parameter $\alpha_i+\alpha_j$ for all $i<j$, $2\tilde{M}_{(i,i)}$ having inverse-gamma distribution with parameter $c+\alpha_i$ 
        (i.e., $\tilde{M}_{(i,i)}$ has density $\frac{2^{-c-\alpha_i}}{\Gamma(c+\alpha_i)} x^{-c-\alpha_i-1} e^{-1/(2x)}$)
        for all $i$, and $\tilde{M}_{(i,j)} = \tilde{M}_{(j,i)}$ for all $i,j$. 
        Define $Z_M$ and $Z_{\tilde{M}}$ as the partition functions on $M$ and $\tilde{M}$. 
    \end{enumerate}

    Then, for any $u \in \zdiag$, the following equality of conditional laws holds:
    \[
    \mathcal L\!\left(
        Z_M|_{\bar V(u)}
        \,\middle|\,
        Z_M|_{D(u)}
    \right)
    =
    \mathcal L\!\left(
        Z_{\tilde M}|_{\bar V(u)}
        \,\middle|\,
        Z_{\tilde M}|_{D(u)}
    \right).
    \]
\end{theorem}

\begin{remark}
\label{remark: half-space directed polymer}
    We note that the half-space directed polymer models are usually defined only within a half-quadrant $H=\{(i,j)\in \N^2:i\geq j\}$ and without any scaling for the diagonal entries. 
    Our setting is closely related to it in the following way:
    let $M$ be a symmetric environment on $\Z^2$, with partition function $Z_M$, and define $M'$ on $H$ by $M'_{(i,j)} = M_{(i,j)}$ for $i> j$ and $M'_{(i,i)} = 2M_{(i,i)}$. Let $Z'_{M'}$ denote the partition function obtained by summing only over directed paths contained in $H$. 
    Then, for $m \geq n \ge 1$, we have 
    \[
    Z'_{M'}(1,1;m,n) 
    = 2Z_M(1,1;m,n),
    \]
    since any path lying within $H$ that touches the diagonal $k$ times corresponds to $2^{k-1}$ distinct paths in the full-space, each with a weight scaled by $2^{-k}$. 
\end{remark}

To prove Theorem~\ref{bridge} from the results in Section~\ref{sec: preliminary}, we also need the isometry properties of the RSK and gRSK maps. 
This essentially says that the last passage values (resp. partition functions) over certain endpoints are preserved under the RSK (resp. gRSK) map, if we put the outputs of the map back into the environment appropriately and make certain modifications of the endpoints. 
For a more explicit statement and proof, see \cite[Theorem~3.3]{dauvergne}. 
Also see \cite{corwin} for the result in the gRSK case.

However, we will need only the following consequence of the isometry properties:

\begin{proposition}
[{\cite[Proposition~3.2]{dauvergne}}]
\label{isometry RSK}
    Let $M$ be any environment, and $Z_M$ be either the last passage value or the partition function.
    In the polymer case, assume $M$ has strictly positive weights.
    Then for any $u \in \zbox$, we can express 
    $Z_M|_{\bar{H}(u)}$ as a function of $Z_M|_{H(u)}$
    and $Z_M|_{\bar{V}(u)}$ as a function of $Z_M|_{V(u)}$. 
\end{proposition}

The following lemma clarifies the translation of the languages used in this section and Section~\ref{sec: preliminary}.

\begin{lemma}
\label{translation}
    Let $u=((x+1,y+1);(x+n,y+n))\in \zbox$, and let
    $ A = (M_{(x+i,y+j)})_{1\leq i,j\leq n}$.
    
    In the LPP case, assume that $M$ has nonnegative integer entries, and let
    $(P,Q)$ be the RSK image of $A$. If $Z_M$ denotes the last passage values, then
    \[
    \sum_{j=1}^k p_{ij}
    =
    Z_M\!\left(((x+1,y+1);(x+n,y+i))^k\right),
    \]
    and
    \[
    \sum_{j=1}^k q_{ij}
    =
    Z_M\!\left(((x+1,y+1);(x+i,y+n))^k\right).
    \]
    
    In the polymer case, assume that $M$ has positive real entries, and let
    $(P,Q)$ be the gRSK image of $A$. If
    $Z_M$ denotes the polymer partition functions, then
    \[
    \prod_{j=1}^k p_{ij}
    =
    Z_M\!\left(((x+1,y+1);(x+n,y+i))^k\right),
    \]
    and
    \[
    \prod_{j=1}^k q_{ij}
    =
    Z_M\!\left(((x+1,y+1);(x+i,y+n))^k\right).
    \]

    In both cases, $P$ encodes $Z_M|_{H(u)}$, $Q$ encodes $Z_M|_{V(u)}$, and their common shape encodes $Z_M|_{D(u)}$. 
\end{lemma}
\begin{proof}
    This is just a restatement in the new notation of our definition of the RSK map, \eqref{eq: RSK def P} and \eqref{eq: RSK def Q}, or the gRSK map, \eqref{eq: gRSK def P} and \eqref{eq: gRSK def Q}.
\end{proof}

In the exponential LPP setting, continuous RSK is a
piecewise-linear bijection from $n\times n$ matrices with
nonnegative real entries to pairs of nonnegative real
Gelfand--Tsetlin patterns with $n$ rows and common shape;
see \cite[Section~8]{oconnell}.
For the $u\in \zbox$ and matrix $A$ in Lemma~\ref{translation},
the LPP identities in that lemma remain valid under continuous
RSK, so $Z_M|_{D(u)}$ encodes the common shape.

\begin{proof}
    [Proof of Theorem~\ref{bridge}] 
    By Proposition~\ref{isometry RSK}, 
    $Z_M|_{\bar V(u)}$ is a deterministic function of
    $Z_M|_{V(u)}$, and similarly for $\tilde M$.
    Hence it suffices to prove
    \[
    \mathcal L\!\left(
        Z_M|_{V(u)}
        \,\middle|\,
        Z_M|_{D(u)}
    \right)
    =
    \mathcal L\!\left(
        Z_{\tilde M}|_{V(u)}
        \,\middle|\,
        Z_{\tilde M}|_{D(u)}
    \right).
    \]
    
   For the geometric LPP case, this follows from Theorem~\ref{bridge LPP} and Lemma~\ref{translation}.
   For the log-gamma polymer case, it follows from Theorem~\ref{bridge polymer} and Lemma~\ref{translation}.

    For the exponential LPP case, translate the indices so that
        $u=(1,1;n,n)$. For $\varepsilon>0$, set
        \[
        \alpha_i^{(\varepsilon)}=e^{-\varepsilon\alpha_i},
        \qquad
        \beta_j^{(\varepsilon)}=e^{-\varepsilon\beta_j},
        \qquad
        c^{(\varepsilon)}=e^{-\varepsilon c},
        \]
        and let $M^{(\varepsilon)}$ be the full-space geometric
        environment with parameter sequences $\alpha^{(\varepsilon)}$
        and $\beta^{(\varepsilon)}$.
        Let $\tilde M^{(\varepsilon)}$ be the half-space geometric
        environment with parameter sequence $\alpha^{(\varepsilon)}$
        and diagonal parameter $c^{(\varepsilon)}$.
        Let $Q^{(\varepsilon)}$ be the second RSK
        pattern of $M^{(\varepsilon)}$ restricted to $u$, and let
        $(R^{(\varepsilon)},R^{(\varepsilon)})$ be the RSK image
        of $\tilde M^{(\varepsilon)}$ restricted to $u$.
        By Propositions~\ref{prop: rsk Q given lambda} and
        \ref{prop: RSK R given lambda}, the conditional laws of
        $\varepsilon Q^{(\varepsilon)}$ given
        $\varepsilon\sh(Q^{(\varepsilon)})$ and of
        $\varepsilon R^{(\varepsilon)}$ given
        $\varepsilon\sh(R^{(\varepsilon)})$ have the same kernel
        $K_\varepsilon$.
        
        For a real triangular array
        $q=(q_{ij})_{1\leq j\leq i\leq n}$, set
        \[
        \Delta_r(q)=\sum_{i=1}^r q_{ri}
        -\sum_{i=1}^{r-1}q_{r-1,i},\qquad 1\leq r\leq n.
        \]
        For $\lambda_1\geq\cdots\geq\lambda_n\geq0$, let
        $\GT_n^{\mathbb R}(\lambda)$ be the set of such arrays
        with last row $\lambda$ satisfying
        $q_{i+1,j+1}\leq q_{ij}\leq q_{i+1,j}$ for
        $1\leq j\leq i<n$.
        For every bounded continuous function $F$ on the real
        triangular arrays with $n$ rows, Riemann-sum convergence
        in the geometric conditional laws gives
        \[
        K_\varepsilon F\bigl(\varepsilon\ell^{(\varepsilon)}\bigr)
        \longrightarrow
        K F(\lambda)
        =
        \frac{
        \displaystyle
        \int_{\GT_n^{\mathbb R}(\lambda)}
        F(q)\exp\left(-\sum_{r=1}^n\alpha_r\Delta_r(q)\right)\,dq
        }{
        \displaystyle
        \int_{\GT_n^{\mathbb R}(\lambda)}
        \exp\left(-\sum_{r=1}^n\alpha_r\Delta_r(q)\right)\,dq
        }
        \]
        whenever $\ell^{(\varepsilon)}$ are partitions with at most
        $n$ parts and $\varepsilon\ell^{(\varepsilon)}\to\lambda$
        with $\lambda_1>\cdots>\lambda_n>0$.
        Here $dq$ is Lebesgue measure in the $n(n-1)/2$ entries
        above the fixed last row. For every other $\lambda\in\R^n$,
        let $K(\lambda,\cdot)$ be the point mass at the zero array.
        
        Let $(P,Q)$ and $(R,R)$ be the continuous RSK images of
        $M$ and $\tilde M$, respectively, restricted to $u$.
        In each environment, the scaled geometric weights on $u$
        converge jointly in distribution to the corresponding
        exponential weights.
        Since the pattern entries are continuous functions of
        finitely many last-passage values, the scaled patterns and
        their shapes converge jointly to the continuous RSK
        patterns and their shapes.
        
        Continuous RSK is a piecewise-linear bijection, and its
        restriction to symmetric nonnegative matrices is a
        piecewise-linear bijection onto single nonnegative real
        Gelfand--Tsetlin patterns. The joint laws of the full-space
        weights on $u$ and of the upper-triangular half-space weights
        on $u$ each have a density. Thus both shape laws are absolutely
        continuous with respect to Lebesgue measure on $\R^n$,
        and both shapes are strictly decreasing and positive almost
        surely.
        
        For every bounded continuous function $G:\R^n\to\R$,
        the geometric conditional-law identity, joint convergence,
        and the preceding convergence of $K_\varepsilon F$ give
        \[
        \mathbb E\left[F(Q)G(\sh(Q))\right]
        =
        \mathbb E\left[G(\sh(Q))K F(\sh(Q))\right].
        \]
        The passage to the limit follows from the extended
        continuous-mapping theorem and boundedness of $F$ and $G$.
        Hence $K$ is a regular conditional distribution of $Q$
        given $\sh(Q)$. Applying the same argument to
        $R^{(\varepsilon)}$ shows that $K$ is also a regular
        conditional distribution of $R$ given $\sh(R)$.
        The LPP identities in Lemma~\ref{translation}, extended
        to continuous RSK, give the required equality.
\end{proof}

\begin{lemma}
\label{lemma: conditioning vector absolute continuity}
Let $(M,Z_M)$ and $(\tilde M,Z_{\tilde M})$ be a pair of
environments as in Theorem~\ref{bridge}. Then, for every
$v\in\zdiag$,
\[
\mathcal L\bigl(Z_{\tilde M}|_{D(v)}\bigr)
\ll
\mathcal L\bigl(Z_M|_{D(v)}\bigr).
\]
\end{lemma}

\begin{proof}
By Lemma~\ref{translation} and its continuous RSK counterpart,
it suffices to show that the half-space RSK or gRSK shape law
is absolutely continuous with respect to the full-space shape
law for the environments restricted to $v$.
Let $n$ be the side length of $v$.

In geometric LPP, the full-space shape law puts positive mass
on every partition with at most $n$ parts, by
\eqref{eq: rsk pushforward mass lpp}. Every half-space shape
is such a partition, so absolute continuity follows.

In exponential LPP, continuous RSK is a piecewise-linear
bijection from nonnegative matrices to pairs of nonnegative
real Gelfand--Tsetlin patterns with common shape; see
\cite[Section~8]{oconnell}. The strictly positive full-space
input density induces a density that is strictly positive
almost everywhere in the interior of the set of pattern pairs.
Integrating out the entries above the common last row gives
a shape density strictly positive almost everywhere on
\[
\{\lambda\in\R^n:\lambda_1>\cdots>\lambda_n>0\}.
\]
The restriction of continuous RSK to symmetric nonnegative
matrices is a piecewise-linear bijection to single nonnegative
real Gelfand--Tsetlin patterns. The joint density of the
independent upper-triangular half-space weights therefore
induces an absolutely continuous law on the pattern entries,
and hence an absolutely continuous shape law concentrated on
the displayed set. This proves the required absolute continuity.

In the log-gamma polymer, the full-space gRSK push-forward
density in Subsection~\ref{subsec: log-gamma polymer} is
strictly positive, and the half-space pattern law in
Proposition~\ref{prop: gRSK pushforward law of R} is
absolutely continuous. Integrating out the entries above the
last row gives a full-space shape density strictly positive
almost everywhere on $(\R_{>0})^n$ and a half-space shape law
absolutely continuous with respect to Lebesgue measure there.
The required absolute continuity follows.
\end{proof}

We will use the next decoupling statement to extend Theorem~\ref{bridge} to more general identities between the half-space and full-space models. 

\begin{lemma}
[{\cite[Proposition~3.1]{dauvergne}}]
\label{decoupling duncan 3.1}
    Suppose $\mathbf{u} = (u_1,\dots, u_k) \in \mathcal{E}$ and $v\in \zbox$ are such that $(u_i, v) \in \mathcal{V} \cup \mathcal{N}$ for all $i$. 
    Then $Z_M(\mathbf{u})$ is a function of the random variables $\{M_x:x\in (u_1 \cup \dots \cup u_k) \setminus v\}$ and the function $Z_M|_{\bar{V}(v)}$. 
\end{lemma}

The proof in \cite{dauvergne} has the function written out explicitly, but this is not needed for us. 

In the next theorems, we will require the two sequences of model parameters $\alpha$ and $\beta$ of the full-space environment to be equal, so that any ``non-symmetric'' part of the half-space environment has the same joint law as that of the corresponding full-space environment.  

\begin{theorem}
\label{general bridge}
    Let $M,Z_M, \tilde{M},Z_{\tilde{M}}$ be a pair of environments as in Theorem~\ref{bridge}, with the additional requirement that $\alpha = \beta$. 
    Suppose $U \subseteq \zbox$ and $(\sqcup U) \cap (\sqcup U)^\top \subseteq v$ for some $v\in \zdiag$,
    and that $\{v\} \times U \subseteq {\mathcal{N}} \cup \mathcal{H}$. 
    Then,
    \[
    \mathcal L\!\left(
        Z_{\tilde{M}}|_{\mathcal{E}_D(U)}
        \,\middle|\,
        Z_{\tilde{M}}|_{D(v)}
    \right)
    =
    \mathcal L\!\left(
        Z_{M}|_{\mathcal{E}_D(U)}
        \,\middle|\,
        Z_{M}|_{D(v)}
    \right).
    \]
\end{theorem}
\begin{proof}
    For $u,w\in \zbox$, 
    write $w \in u^k$ if $w = (u^-+(j,0);u^+-(k-1-j,0))$ for some $0 \leq j < k$. 

    Notice that if $u \in \zbox$ is such that $(u,v) \in \mathcal{V}$ (resp. $(u,v)\in \mathcal{N}$), 
    then for each 
    $ u^k \in D(u)$
    and each $w \in u^k$,
    we also have $(w,v) \in \mathcal{V}$ (resp. $(w,v) \in \mathcal{N}$). 
    
    For any $\mathbf{u} \in \mathcal{E}_D(U)$, we can write it as $\mathbf{u} = (u_1,u_2,\dots, u_m)$ where each $u_i \in w_i^{k_i}$ for some $w_i\in U$ and $w_i^{k_i} \in D(w_i)$.
    Also, each $u_i \subseteq w_i \subseteq  \sqcup U$, so $(u_1\cup \dots \cup u_m ) \setminus v  \subseteq \sqcup U \setminus v$.
    
    Hence by applying Lemma~\ref{decoupling duncan 3.1}, each $Z_M(\mathbf{u})$ is a function of $\{M_x:x\in \sqcup U \setminus v\}$ and $Z_M|_{\bar{V}(v)}$. 
    Consequently, $Z_M|_{\mathcal{E}_D(U)}$ itself is a function of $\{M_x:x\in \sqcup U \setminus v\}$ and $Z_M|_{\bar{V}(v)}$.
    Similarly, $Z_{\tilde{M}}|_{\mathcal{E}_D(U)}$ is the same function of $\{\tilde{M}_x:x\in \sqcup U \setminus v\}$ and $Z_{\tilde{M}}|_{\bar{V}(v)}$.

    Clearly $Z_M|_{\bar{V}(v)}$ and $Z_M|_{D(v)}$ only depend on $\{M_x:x\in v\}$, so they are independent from $\{M_x:x\in \sqcup U \setminus v\}$. 
    Similarly,  $Z_{\tilde{M}}|_{\bar{V}(v)}$ and $Z_{\tilde{M}}|_{D(v)}$ are independent from $\{\tilde{M}_x:x\in \sqcup U \setminus v\}$.

    Hence it remains to verify two facts:
    \begin{equation}
    \label{eq: general bridge 1}
        \mathcal L\!\left(
            Z_{\tilde M}|_{\bar V(v)}
            \,\middle|\,
            Z_{\tilde M}|_{D(v)}
        \right)
        =
        \mathcal L\!\left(
            Z_M|_{\bar V(v)}
            \,\middle|\,
            Z_M|_{D(v)}
        \right)
    \end{equation}
    and
    \begin{equation}
    \label{eq: general bridge 2}
        \{\tilde M_x:x\in \sqcup U\setminus v\}
        \disteq
        \{M_x:x\in \sqcup U\setminus v\}.
    \end{equation}

    The identity in \eqref{eq: general bridge 1} is exactly Theorem~\ref{bridge},
    while \eqref{eq: general bridge 2} follows from the independence of the two collections and equality of their corresponding marginal distributions.
    More explicitly, note that if $(i,j) \in \sqcup U \setminus v$, then $(j,i) \notin \sqcup U \setminus v$, as otherwise $(i,j) \in (\sqcup U) \cap (\sqcup U)^\top \setminus v$, contradicting that $(\sqcup U) \cap (\sqcup U)^\top \subseteq v$, so the random variables $\tilde{M}_x:x\in \sqcup U \setminus v$ are independent. 
    This also means that $(i,i) \notin \sqcup U \setminus v$ for all $i$. But by our definition of the environments, $M_{(i,j)} \disteq \tilde{M}_{(i,j)}$ for any $i\neq j$, so $M_x \disteq \tilde{M}_x$ for any $x \in \sqcup U \setminus v$.   
\end{proof}

\begin{remark}
\label{remark: auxiliary full-space parameters}
    The condition $\alpha=\beta$ in
    Theorem~\ref{general bridge} is only a convenient sufficient condition
    for \eqref{eq: general bridge 2}. The proof in fact applies to any
    auxiliary full-space environment for which
    \eqref{eq: general bridge 1} holds on $v$ and whose entries on
    $\sqcup U\setminus v$ have the same joint law as the corresponding
    half-space entries. The finite-matrix proof of Theorem~\ref{bridge}
    shows that the conditional kernel on $v$ depends on the first
    full-space parameter sequence only through its restriction to the
    indices of $v$, and is independent of the second full-space parameter
    sequence. Thus one may first fix the half-space environment on $v$,
    and then choose the auxiliary full-space parameters after $v$ and $U$
    have been specified.
\end{remark}

Theorem~\ref{general bridge} allows us to directly transfer distributional invariances in the full space into those in the corresponding half-space environments. 

For a function $f: S\to T$, where $S,T\subseteq \zbox$, we define 
\[
\bar{f}: \mathcal{E}_D(S) \to \mathcal{E}_D(T)
\]
for its extension such that 
\[
\bar{f}(u_1^{k_1},\dots, u_l^{k_l}) = (f(u_1)^{k_1},\dots, f(u_l)^{k_l}).
\]

For well-definedness, we assume all functions in our consideration satisfy the following:
for every representation
$\mathbf u=(u_1^{k_1},\dots,u_\ell^{k_\ell})
\in\mathcal E_D(S),$
the vector
$(f(u_1)^{k_1},\dots,f(u_\ell)^{k_\ell})$
belongs to $\mathcal E_D(T)$ and is independent of the chosen
representation of $\mathbf u$.

\begin{theorem}
\label{transferring invariance}
     Let $M,Z_M, \tilde{M},Z_{\tilde{M}}$ be a pair of environments as in Theorem~\ref{bridge}, with the additional requirement that $\alpha = \beta$. 
    Suppose $U , U'\subseteq \zbox$ and $v \in \zdiag$ are such that $\sqcup U \cap (\sqcup U)^\top, \sqcup U' \cap (\sqcup U')^\top \subseteq v$
    and that $\{v\} \times U, \{v\} \times U' \subseteq {\mathcal{N}} \cup \mathcal{H}$.  
    
    Suppose in addition that there exists a function 
    $f :U \cup \{v\} \to U' \cup \{v\}$ 
    such that $f(v)=v$, $f(U)\subseteq U'$, and 
    \[
    Z_M |_{\mathcal{E}_D(U\cup \{v\})} \disteq
    Z_M|_{\mathcal{E}_D(U'\cup \{v\})} \circ \bar{f},
    \]
    then
    \[
    Z_{\tilde{M}} |_{\mathcal{E}_D(U)} \disteq
    Z_{\tilde{M}}|_{\mathcal{E}_D(U')} \circ \bar{f}.
    \]
\end{theorem}

\begin{proof}
    Note that both $\mathcal{E}_D(U)$ and $D(v)$ are subsets of $\mathcal{E}_D(U \cup \{v\})$, 
    both $\mathcal{E}_D(U')$ and $D(v)$ are subsets of $\mathcal{E}_D(U' \cup \{v\})$, 
    and $\bar{f}$ restricts to the identity on $D(v)$.
    Hence, 
    the assumed distributional identity implies
    \begin{equation}
    \label{eq: full-space conditional invariance}
        \mathcal L\!\left(
            Z_M|_{\mathcal{E}_D(U)}
            \,\middle|\,
            Z_M|_{D(v)}
        \right)
        =
        \mathcal L\!\left(
            Z_M|_{\mathcal{E}_D(U')} \circ \bar{f}
            \,\middle|\,
            Z_M|_{D(v)}
        \right).
    \end{equation}
    Choose the two conditional laws in
    \eqref{eq: full-space conditional invariance} to be
    the common full-space and half-space versions supplied by
    Theorem~\ref{general bridge}, applied to $U$ and $U'$.
    Equation~\eqref{eq: full-space conditional invariance} holds for
    $\mathcal L(Z_M|_{D(v)})$-almost every conditioning value.
    By Lemma~\ref{lemma: conditioning vector absolute continuity},
    it also holds for
    $\mathcal L(Z_{\tilde M}|_{D(v)})$-almost every conditioning
    value. Using these versions for the half-space conditional
    laws gives
    \[
    \mathcal L\!\left(
        Z_{\tilde{M}}|_{\mathcal{E}_D(U)}
        \,\middle|\,
        Z_{\tilde{M}}|_{D(v)}
    \right)
    =
    \mathcal L\!\left(
        Z_{\tilde{M}}|_{\mathcal{E}_D(U')} \circ \bar{f}
        \,\middle|\,
        Z_{\tilde{M}}|_{D(v)}
    \right),
    \]
    which implies
    \[
    Z_{\tilde{M}} |_{\mathcal{E}_D(U)} \disteq
    Z_{\tilde{M}}|_{\mathcal{E}_D(U')} \circ \bar{f}.
    \]
\end{proof}

Using Theorem~\ref{transferring invariance} and invariance results in the full-space environments (for example those proved in \cite{dauvergne}), we will be able to prove various invariance results related to shifting boxes in the half-space environments. 
However, Theorem~\ref{transferring invariance} does not deal with the invariances in which the environment itself changes (i.e. the defining parameters change), for which we will need a slightly stronger theorem.

To make the choice of parameters explicit, we will write a pair of environments $M$ and $\tilde{M}$ defined as in Theorem~\ref{bridge} using parameters $\alpha, \beta, c$ as 
$M_{\alpha,\beta}$ and $\tilde{M}_{\alpha,c}$. 
Any two environments defined with different parameters will be assumed to be independent. 

First, we establish the following distributional identity for diagonal boxes. 

\begin{proposition}
\label{diagonal change of parameter}
    Let $\tilde{M}_{\alpha,c}$ and $\tilde{M}_{\alpha',c}$ be two half-space environments of the same type as in Theorem~\ref{bridge}, and let $v \in \zdiag$. 
    Suppose that $(\alpha'_i:v^-_1\leq i \leq v^+_1)$ is a permutation of $(\alpha_i:v^-_1\leq i \leq v^+_1)$, then 
    \[
    Z_{\tilde{M}_{\alpha,c}}|_{D(v)} \disteq Z_{\tilde{M}_{\alpha',c}}|_{D(v)}.
    \]
\end{proposition}
\begin{proof}
    The geometric LPP case follows from Proposition~\ref{lpp matrix change of parameter} and Lemma~\ref{translation}. 
    The log-gamma polymer case follows from Proposition~\ref{polymer matrix change of parameter} and Lemma~\ref{translation}.
    The exponential LPP case follows as a limit of the geometric LPP case. 
\end{proof}

We can now state the next theorem for transferring invariances of inhomogeneous environments with a permutation of parameters.

\begin{theorem}
\label{transferring invariance inhomogeneous}
    Let $M_{\alpha,\alpha}, \tilde{M}_{\alpha,c}, M_{\alpha',\alpha'},\tilde{M}_{\alpha',c}$ be two pairs of environments of the same types as in Theorem~\ref{bridge}. 
    Suppose $U , U'\subseteq \zbox$ and $v \in \zdiag$ are such that $\sqcup U \cap (\sqcup U)^\top, \sqcup U' \cap (\sqcup U')^\top \subseteq v$
    and that $\{v\} \times U, \{v\} \times U' \subseteq {\mathcal{N}} \cup \mathcal{H}$, 
    and that $(\alpha'_i:v^-_1\leq i \leq v^+_1)$ is a permutation of $(\alpha_i:v^-_1\leq i \leq v^+_1)$.
    Suppose further that there exists a function 
    $f :U \cup \{v\} \to U' \cup \{v\}$ 
    such that $f(v)=v$, $f(U)\subseteq U'$, and
    \[ Z_{M_{\alpha,\alpha}} |_{\mathcal{E}_D(U\cup \{v\})} \disteq
    Z_{M_{\alpha',\alpha'}}|_{\mathcal{E}_D(U'\cup \{v\})} \circ \bar{f},\]
    then
    \[ Z_{\tilde{M}_{\alpha,c}} |_{\mathcal{E}_D(U)} \disteq
    Z_{\tilde{M}_{\alpha',c}}|_{\mathcal{E}_D(U')} \circ \bar{f}.\]
\end{theorem}
\begin{proof}
    The assumed distributional identity implies
    \begin{equation}
    \label{eq: inhomogeneous full-space conditional invariance}
        \mathcal L\!\left(
            Z_{M_{\alpha,\alpha}}|_{\mathcal{E}_D(U)}
            \,\middle|\,
            Z_{M_{\alpha,\alpha}}|_{D(v)}
        \right)
        =
        \mathcal L\!\left(
            Z_{M_{\alpha',\alpha'}}|_{\mathcal{E}_D(U')} \circ \bar{f}
            \,\middle|\,
            Z_{M_{\alpha',\alpha'}}|_{D(v)}
        \right).
    \end{equation}
    The assumed distributional identity also implies that
    $Z_{M_{\alpha,\alpha}}|_{D(v)}$ and
    $Z_{M_{\alpha',\alpha'}}|_{D(v)}$ have the same law, since
    $\bar f$ is the identity on $D(v)$.
    Choose the two conditional laws in
    \eqref{eq: inhomogeneous full-space conditional invariance}
    to be the common full-space and half-space versions supplied by
    Theorem~\ref{general bridge}, applied to the pair
    $(M_{\alpha,\alpha},\tilde M_{\alpha,c})$ on $U$ and to the
    pair $(M_{\alpha',\alpha'},\tilde M_{\alpha',c})$ on $U'$.
    Equation~\eqref{eq: inhomogeneous full-space conditional invariance}
    holds almost everywhere with respect
    to the common full-space conditioning law.
    By Lemma~\ref{lemma: conditioning vector absolute continuity},
    the laws of $Z_{\tilde M_{\alpha,c}}|_{D(v)}$ and
    $Z_{\tilde M_{\alpha',c}}|_{D(v)}$ are both absolutely
    continuous with respect to that law. Thus the equality also
    holds almost everywhere with respect to both half-space
    conditioning laws. Using these versions for the half-space
    conditional laws gives
    \[
    \mathcal L\!\left(
        Z_{\tilde{M}_{\alpha,c}}|_{\mathcal{E}_D(U)}
        \,\middle|\,
        Z_{\tilde{M}_{\alpha,c}}|_{D(v)}
    \right)
    =
    \mathcal L\!\left(
        Z_{\tilde{M}_{\alpha',c}}|_{\mathcal{E}_D(U')} \circ \bar{f}
        \,\middle|\,
        Z_{\tilde{M}_{\alpha',c}}|_{D(v)}
    \right).
    \]
    By Proposition~\ref{diagonal change of parameter}, the marginal laws of 
    $Z_{\tilde{M}_{\alpha,c}}|_{D(v)}$ and
    $Z_{\tilde{M}_{\alpha',c}}|_{D(v)}$ are the same.
    Therefore, we have 
    \[
    Z_{\tilde{M}_{\alpha,c}}|_{\mathcal{E}_D(U)}
    \disteq
    Z_{\tilde{M}_{\alpha',c}}|_{\mathcal{E}_D(U')} \circ \bar{f}.
    \]
\end{proof}

\subsection{Invariance in homogeneous half-space models}

In this subsection, we will assume $M,Z_M,\tilde{M},Z_{\tilde{M}}$ is a pair of environments as in Theorem~\ref{bridge}, such that $\alpha_i = \beta_i = p$ for some $p$ for all $i\in \Z$, so that all the entries in $M$ and all the entries in $\tilde{M}$ strictly below diagonal are i.i.d. 
Note that the diagonal entries in $\tilde{M}$ are also i.i.d., but the distribution may be different from the off-diagonal entries due to the additional parameter $c$ involved, which can be chosen arbitrarily subject to the
admissibility conditions in Theorem~\ref{bridge}.

We record the following invariance properties in the full-space environments.
For $c \in \Z^2$, recall that $T_c$ is the translation by $c$, i.e. $T_c u = u+(c;c)$ for any $u \in \zbox$.

\begin{proposition}
[{\cite[Theorem~4.2 and Proposition~4.13]{dauvergne}}]
\label{duncan tower} 
    Let $(M,Z_M)$ be a homogeneous full-space environment of one of the types considered in
    Theorem~\ref{bridge}.
    Then, the following shift-invariances hold.
    
    \begin{enumerate}
        \item \textup{(Towers)}
        Let $U_1,\ldots,U_k, \,U'_1,\ldots,U'_k\subseteq \zbox$ be
        nonempty sets such that
        $U_i\times U_{i+1},\, U'_i\times U'_{i+1}\subseteq \mathcal H$
        for each $1\leq i<k$. 
        Suppose also that for each $i$, there exists $c_i\in\mathbb Z^2$
        such that
        $U'_i=T_{c_i}U_i$.
        Let
        $f:\bigcup U_i\to \bigcup U'_i$
        be the map whose restriction to $U_i$ is $T_{c_i}$. 
        Then
        \[
        Z_M|_{\mathcal E_D(\bigcup U_i)}
        \disteq
        Z_M|_{\mathcal E_D(\bigcup U'_i)}\circ \bar f .
        \]
        
        \item \textup{(Box permutations)}
        Let
        $U_1,\ldots,U_k,V_1,\ldots,V_m,
        U'_1,\ldots,U'_k,V'_1,\ldots,V'_m
        \subseteq \zbox$
        be such that
        $V_i\times U_j,
        V'_i\times U'_j\subseteq \mathcal H$
        for all $i,j$, and
        $U_i\times U_j,\; V_i\times V_j,\; U'_i\times U'_j,\; V'_i\times V'_j
        \subseteq \mathcal N$ for all $i\neq j$.
        Suppose also that for each $i,j$, there exist
        $c_i,d_j\in\mathbb Z^2$ such that
        $U'_i=T_{c_i}U_i$ and $V'_j=T_{d_j}V_j$ .
        Let
        $f:\bigcup U_i \cup
          \bigcup V_j
        \to
        \bigcup U'_i\cup
        \bigcup V'_j$
        be the map whose restriction to $U_i$ is $T_{c_i}$, and whose restriction to
        $V_j$ is $T_{d_j}$. Then
        \[
        Z_M|_{\mathcal E_D(\bigcup U_i\cup\, \bigcup V_j)}
        \disteq
        Z_M|_{\mathcal E_D(\bigcup U'_i\cup\,\bigcup V'_j)}
        \circ \bar f .
        \]
    \end{enumerate}
\end{proposition}

Based on these, we can get analogous results for the half-space environments.

\begin{theorem}
    [Half-space tower shift] 
    \label{tower shift}
    Let $U_1, \dots ,U_k, U_1', \dots, U_k' \subseteq \zbox$ be nonempty sets such that
    \begin{enumerate}
        \item There exists $v \in \zdiag$, so that $\{v\} \times U_1 , \{v\} \times U'_1\subseteq \mathcal{H}$;
        \item $U_i \times U_{i+1}, U'_i \times U'_{i+1} \subseteq \mathcal{H}$ for every $1\leq i < k$;
        \item For each $i$, $U'_i = T_{c_i} U_i$ for some $c_i \in \Z^2$.
    \end{enumerate}
    Let $f: \bigcup U_i \to \bigcup U'_i$ be the function whose restriction to $U_i$ is $T_{c_i}$. Then 
    \[
    Z_{\tilde{M}} |_{\mathcal{E}_D(\bigcup U_i)} \disteq
    Z_{\tilde{M}}|_{\mathcal{E}_D(\bigcup U_i')} \circ \bar{f}.\]
    
\end{theorem} 
\begin{proof}
    Since every layer is nonempty and horizontal crossing is transitive, the first two conditions imply
    $\{v\}\times U_i,\{v\}\times U_i'\subseteq\mathcal H$ for every $i$.
    Extend $f$ to a function from $\{v\} \cup \bigcup U_i$ to $\{v\} \cup \bigcup U_i'$ by letting $f(v) = v$. Apply Proposition~\ref{duncan tower}(1) to 
    $\{v\}, U_1,\dots U_k$ and $\{v\}, U'_1,\dots, U'_k$ to get 
     \[
    Z_{{M}} |_{\mathcal{E}_D(\bigcup U_i\cup \{v\})} \disteq
    Z_{{M}}|_{\mathcal{E}_D(\bigcup U_i'\cup \{v\})} \circ \bar{f}.\]
    Then the result follows from Theorem~\ref{transferring invariance} with $U = \bigcup U_i$ and $U' = \bigcup U_i'$. 
    To see that $\sqcup U \cap (\sqcup U)^\top \subseteq v$ and $\sqcup U' \cap (\sqcup U')^\top \subseteq v$, 
    notice that if $(i,j) \in \sqcup U \setminus v$ or $\sqcup U' \setminus v$ then $i \in [v^-_1, v^+_1]$ and $j \notin [v^-_1, v^+_1]$, so it's impossible to also have $(j,i) \in \sqcup U \setminus v$ or $\sqcup U' \setminus v$. 
\end{proof}

\begin{theorem}
    [Half-space disjoint shift]
    \label{disjoint shift}
    Let $V, U_1, \dots, U_k, U'_1,\dots, U'_k \subseteq \zbox$ be such that
    \begin{enumerate}
        \item There exists $v \in \zdiag$, so that $\{v\} \times (V \cup \bigcup U_i \cup \bigcup U'_i) \subseteq \mathcal{H}$;
        \item For all $i$, $V \times U_i, V \times U'_i \subseteq \mathcal{H}$;
        \item For all $i\neq j$, $U_i \times U_j, U'_i \times U'_j \subseteq \mathcal{N}$;
        \item For each $i$, $U'_i = T_{c_i} U_i$ for some $c_i \in \Z^2$.
    \end{enumerate}
    Let $f: V \cup \bigcup U_i \to V \cup \bigcup U'_i$ be the function whose restriction to $U_i$ is $T_{c_i}$ and whose restriction to $V$ is the identity. Then 
    \[
    Z_{\tilde{M}} |_{\mathcal{E}_D(V \cup \bigcup U_i)} \disteq
    Z_{\tilde{M}}|_{\mathcal{E}_D(V \cup \bigcup U_i')} \circ \bar{f}.
    \]
\end{theorem}
\begin{proof}
    Extend $f$ to a function from $\{v\} \cup V \cup  \bigcup U_i$ to $\{v\} \cup V \cup \bigcup U'_i$ by setting $f(v) = v$. 
    Apply Proposition~\ref{duncan tower}(2) with $m=1$, $V_1 = V_1' = \{v\} \cup V$, and $d_1 =(0,0)$ to get
    \[
    Z_{{M}} |_{\mathcal{E}_D(\{v\} \cup V \cup \bigcup U_i)} \disteq
    Z_{{M}}|_{\mathcal{E}_D(\{v\} \cup V \cup \bigcup U_i')} \circ \bar{f}.
    \]
    Then the result follows from Theorem~\ref{transferring invariance} with $U = V \cup \bigcup U_i$ and $U' = V \cup \bigcup U'_i$. That $\sqcup U \cap (\sqcup U)^\top \subseteq v$ and $\sqcup U' \cap (\sqcup U')^\top \subseteq v$ are seen in the same way as in the previous proposition.
\end{proof}

\subsection{Invariance of inhomogeneous half-space models}

In this subsection, we return to the general setting of environments $M , Z_M, \tilde{M}, Z_{\tilde{M}}$ as in Theorem~\ref{bridge} without requiring the parameters to be homogeneous.

We first record a known invariance result for permuting parameters and shifting boxes in the full-space model. 
Following \cite{dauvergne}, we define for $\alpha$ a bi-infinite sequence and $u\in \zbox$ that 
\[ 
\alpha_u = (\alpha_i:(i,j)\in u \text{ for some }j)
\quad \text{ and } \quad 
\hat\alpha_u = (\alpha_j:(i,j)\in u \text{ for some } i).
\]

Further, for $U \subseteq \zbox$, define 
\[
I(U)^-_1
\]
to be the smallest interval containing $\{u^-_1:u\in U\}$, and similarly define $I(U)^+_1, I(U)^-_2,$ and $I(U)^+_2$. 

Then, define the completion of $U$ by
\[
\overline U \defeq \{u\in \zbox: u^-_1 \in I(U)^-_1, u^-_2 \in I(U)^-_2, u^+_1 \in I(U)^+_1, u^+_2 \in I(U)^+_2\}.
\]

We will only consider families $U_1, \dots, U_k$ with prescribed translations $T_{c_1},\dots, T_{c_k}$ such that completed layers translated by different vectors are disjoint, i.e.
\[
\overline{U_i} \cap \overline{U_j} \neq \emptyset \implies c_i=c_j.
\]
Thus, if $f|_{U_i} = T_{c_i}$ for each $i$, then $f$ extends unambiguously to $\bigcup_i \overline{U_i}$ by setting $f|_{\overline U_i} = T_{c_i}$. 

This extension of $f$ to the completed sets should not be confused with the map $\bar f$ defined above, which denotes the induced map on the extended endpoint sets $\mathcal E_D(\cdot)$. 

\begin{proposition}
[{\cite[Theorem~4.17]{dauvergne}}]
    \label{duncan parameter shift}
    Let $U_1, \dots, U_k, U'_1, \dots, U'_k \subseteq \zbox$ be 
    nonempty sets such that
    $U_i \times U_{i+1}, \; U'_i \times U'_{i+1} \subseteq \mathcal{H}$ for every $1 \leq i < k$, 
    and each $U'_i = T_{c_i} U_i$ for some $c_i \in \Z^2$.
    Let $f: \bigcup U_i \to \bigcup U'_i$ be the function whose restriction to $U_i$ is $T_{c_i}$. 
    Let $\alpha, \beta, \alpha', \beta'$ be bi-infinite sequences such that for each $u\in \bigcup \overline{U_i}$, there exist permutations $\sigma_u$ and $\tau_u$ such that
    \[
        \alpha_u = \sigma_u ( \alpha'_{f(u)}) \quad \text{and} \quad \hat{\beta}_u = \tau_u ( \hat{\beta}'_{f(u)}).
    \]
    Then,
    \[ 
    Z_{M_{\alpha, \beta}}|_{\mathcal{E}_D(\bigcup U_i)} \disteq 
    Z_{M_{\alpha', \beta'}}|_{\mathcal{E}_D(\bigcup U_i')} \circ \bar{f}.
    \]
\end{proposition}

This yields the following half-space counterpart.
We restrict the
parameters so that the full-space environments $M_{\alpha,\alpha}$
and $M_{\alpha',\alpha'}$ used in the proof are well-defined.

\begin{theorem}
\label{inhomogeneous shift and permute}
    Let $U_1, \dots ,U_k, U_1', \dots, U_k' \subseteq \zbox$ be nonempty sets such that
    \begin{enumerate}
        \item There exists $v,v' \in \zdiag$, so that $\{v\} \times U_1 , \{v'\} \times U'_1\subseteq \mathcal{H}$;
        \item $U_i \times U_{i+1}, U'_i \times U'_{i+1} \subseteq \mathcal{H}$ for every $1\leq i <k$;
        \item For each $i$, $U'_i = T_{c_i} U_i$ for some $c_i \in \Z^2$.
    \end{enumerate}
    Let $f: \bigcup U_i \to \bigcup U'_i$ be the function whose restriction to $U_i$ is $T_{c_i}$. 
    Let $\alpha, \alpha'$ be bi-infinite sequences 
    with $\alpha_i,\alpha_i' \in (0,1)$ for all $i$ in the geometric case, and $\alpha_i, \alpha_i' >0$ for all $i$ in the exponential and log-gamma cases. 
    Suppose that
    for all $u\in \bigcup \overline{U_i}$, there exist permutations $\sigma_u$ and $\tau_u$ such that 
    \[
    \alpha_u = \sigma_u ( \alpha'_{f(u)}) \quad \text{and} \quad \hat{\alpha}_u = \tau_u ( \hat{\alpha}'_{f(u)}).
    \]
    Then, for every $c$ for which both
    $\tilde M_{\alpha,c}$ and $\tilde M_{\alpha',c}$ are well-defined, we have
    \[
    Z_{\tilde{M}_{\alpha, c}}|_{\mathcal{E}_D(\bigcup U_i)} \disteq 
    Z_{\tilde{M}_{\alpha', c}}|_{\mathcal{E}_D(\bigcup U_i')} \circ \bar{f}.
    \]
\end{theorem}
\begin{proof}
    The parameter assumptions ensure that $M_{\alpha,\alpha}$ and $M_{\alpha',\alpha'}$ are well-defined.
    
    Notice that if $\alpha,\beta$ are parameter sequences with $\beta_i = \alpha_{i-r}$ for some $r \in \Z$, then 
    \[
    Z_{\tilde{M}_{\alpha,c}} \disteq Z_{\tilde{M}_{\beta,c}} \circ T_{(r,r)}.
    \]
    Using this, upon shifting one of the environments entirely along the diagonal,
    we may assume that $c_1$ is of the form $(0,y)$,
    i.e. $U_1'$ is obtained from $U_1$ by a vertical shift. 
    
    Then, we can make a particular choice of $v=v'\in \zdiag$ in condition (1), such that $v^-_1=v^-_2 = \min\{u^-_1:u\in U_1\}$ 
    and $v^+_1=v^+_2 = \max\{u^+_1:u\in U_1\}$.
    This still ensures that $\{v\} \times U_1, \{v\} \times U'_1 \subseteq \mathcal{H}$.
    Since every layer is nonempty and horizontal crossing is transitive, $v$ horizontally crosses every $U_i$ and $U_i'$, and hence their completions.
    Also, by definition of $\overline{U_1}$, there exists $u \in \overline{U_1}$ with $u^-_1 = v_1^-$ and $u^+_1 = v^+_1$, 
    so that $\alpha_v = \alpha_u$ and $\alpha'_{f(u)} = \alpha'_{v}$.
    Let $\sigma_u$ be the permutation such that $\alpha_u = \sigma_u(\alpha'_{f(u)})$. 
    Then, we also have $\alpha_v = \sigma_u(\alpha'_{v})$ and $\hat\alpha_v = \alpha_v = \sigma_u(\alpha'_{v}) = \sigma_u(\hat\alpha'_v)$.

    Setting $f(v)=v$ agrees with the prescribed rule
    $f|_{\overline{U_i}}=T_{c_i}$ for every $i$.
    Indeed, if $v\in\overline{U_i}$, then
    $T_{c_i}v\in\overline{U_i'}$ is horizontally crossed by $v$.
    Since $v$ and $T_{c_i}v$ have the same side lengths,
    this forces $c_i=(0,0)$, and hence $T_{c_i}v=v$.
    Hence, we can apply Proposition~\ref{duncan parameter shift} to the towers 
    $\{v\}, U_1,\dots,U_k$ and $\{v\},U'_1,\dots, U'_k$ with parameter sequences $\alpha,\alpha, \alpha',\alpha'$ and $f$ extended as the identity on $\{v\}$, 
    resulting in the identity
    \[
    Z_{M_{\alpha,\alpha}}|_{\mathcal{E}_D(\bigcup U_i \cup \{v\})} 
    \disteq 
    Z_{M_{\alpha',\alpha'}}|_{\mathcal{E}_D(\bigcup U_i' \cup \{v\})} \circ \bar{f}.
    \]
    
    The result then follows from Theorem~\ref{transferring invariance inhomogeneous},
    since the required inclusions $\sqcup U\cap(\sqcup U)^\top,\sqcup U'\cap(\sqcup U')^\top\subseteq v$, with $U=\bigcup U_i$ and $U'=\bigcup U_i'$, follow exactly as in the proof of Theorem~\ref{tower shift}.
\end{proof}

\subsection{A conditional independence property}
\label{subsec: half-space conditional independence}

The following conditional independence property was pointed out to us
by Jimmy He. We prove it by transferring the
corresponding full-space decoupling property through
Theorem~\ref{general bridge}.

The result closely parallels the full-space decoupling property in
\cite[Theorem~2.2]{dauvergne}. In that result,
$Z_M|_{H(u)}$ and $Z_M|_{V(u)}$, which encode the two RSK
output patterns associated with the environment inside $u$, are
conditionally independent given their common shape
$Z_M|_{D(u)}$. Together with
\cite[Proposition~3.2]{dauvergne}, this implies the stronger
full-space relation
\begin{equation}
\label{eq: full-space extended conditional independence}
Z_M|_{\bar H(u)}
\perp\!\!\!\perp
Z_M|_{\bar V(u)}
\;\mid\;
Z_M|_{D(u)}.
\end{equation}

We first record a short conditioning observation that will be used to
transfer this relation.

\begin{lemma}
\label{lemma: conditional independence under marginal change}
Let $E_0,E_1,E_2,E_3$ be standard Borel spaces, and let
$\mathbb P$ and $\widetilde{\mathbb P}$ be probability measures on
$E_0\times E_1\times E_2\times E_3$, with coordinate variables
$(V,X,Y,W)$. Suppose that
\[
\frac{d\widetilde{\mathbb P}}{d\mathbb P}=\rho(V)
\]
for some nonnegative measurable function $\rho$. If, under
$\mathbb P$,
\[
(V,X)\perp\!\!\!\perp W \;\mid\; Y,
\]
then the same conditional independence holds under
$\widetilde{\mathbb P}$.
\end{lemma}

\begin{proof}
Write $\mathbb P^Y$ for the law of $Y$ under $\mathbb P$, let
$\mathbb P_y$ be a regular conditional law of $(V,X,W)$ given
$Y=y$, and denote its $(V,X)$- and $W$-marginals by
$\mathbb P_y^{V,X}$ and $\mathbb P_y^W$. The assumed conditional
independence says that, for $\mathbb P^Y$-almost every $y$,
\[
\mathbb P_y(dv,dx,dw)
=
\mathbb P_y^{V,X}(dv,dx)\,\mathbb P_y^W(dw).
\]
Set
\[
m(y)
=
\int \rho(v)\,\mathbb P_y^{V,X}(dv,dx).
\]
Then $m(Y)$ is a version of
$\mathbb E[\rho(V)\mid Y]$.
Here $\mathbb E$ and $\widetilde{\mathbb E}$ denote expectations
under $\mathbb P$ and $\widetilde{\mathbb P}$, respectively.
Conditional Bayes' formula gives, for every bounded measurable
function $h$,
\[
\widetilde{\mathbb E}[h(V,X,W)\mid Y]
=
\frac{
\mathbb E[\rho(V)h(V,X,W)\mid Y]
}{
 m(Y)
}.
\]
Let $\widetilde{\mathbb P}_y$ be a regular conditional law of
$(V,X,W)$ given $Y=y$ under $\widetilde{\mathbb P}$, and let
$\widetilde{\mathbb P}_y^{V,X}$ be its $(V,X)$-marginal. Then, for
$m(y)>0$,
\begin{align*}
\widetilde{\mathbb P}_y(dv,dx,dw)
&=
\frac{\rho(v)}{m(y)}
\mathbb P_y^{V,X}(dv,dx)\,\mathbb P_y^W(dw)\\
&=
\widetilde{\mathbb P}_y^{V,X}(dv,dx)\,\mathbb P_y^W(dw).
\end{align*}
The law of $Y$ under $\widetilde{\mathbb P}$ has density $m$
with respect to $\mathbb P^Y$, and hence
\[
\widetilde{\mathbb P}(m(Y)=0)
=
\int_{\{y:m(y)=0\}}m(y)\,\mathbb P^Y(dy)
=0.
\]
Thus the displayed factorization holds for
$\widetilde{\mathbb P}$-almost every conditioning value, proving
that $(V,X)$ and $W$ remain independent given $Y$ under
$\widetilde{\mathbb P}$.
\end{proof}

We now define the horizontal endpoint set appearing in the result.
Let $u\in\zcross$.
Define
\[
HC(u)
=
\{
    (w_1^{k_1},\ldots,w_m^{k_m})\in\mathcal E:
    m, k_1,\ldots,k_m \in \Z_{>0},\ 
    w_i\in\zcross\cap\bar H(u)
    \text{ for all } i
\}.
\]
Thus, $HC(u)$ consists of valid concatenations of multi-point
passages $w_i^{k_i}$, where each underlying box $w_i$ spans $u$
horizontally and crosses the diagonal vertically.

\begin{theorem}
\label{half-space conditional independence}
Let $(\tilde M,Z_{\tilde M})$, with parameters $\alpha,c$, be one of the three half-space
environments described in Theorem~\ref{bridge}. 
Assume in addition that $\alpha_i \in (0,1)$ for all $i$ in the geometric case, and $\alpha_i >0$ for all $i$ in the exponential and log-gamma cases. 
Then, for every
$u\in\zcross$,
\[
Z_{\tilde M}|_{HC(u)}
\perp\!\!\!\perp
Z_{\tilde M}|_{\bar V(u)}
\;\mid\;
Z_{\tilde M}|_{D(u)}.
\]
\end{theorem}

\begin{proof}
Let
$v=(u^-_1,u^-_1;u^+_1,u^+_1)$
be the largest diagonal square contained in $u$, and let
$(M,Z_M)$ be the corresponding full-space environment with both
parameter sequences equal to the half-space parameter sequence
$\alpha$.
The parameter assumptions in the theorem ensure that
this full-space environment is well-defined.

For $w=(w^-;w^+)$ and $w^k\in D(w)$, let
\[
\widehat w^{\,k}
=
\bigl(
(w^-;w^+-(0,k-1)),
(w^-+(0,1);w^+-(0,k-2)),
\dots,
(w^-+(0,k-1);w^+)
\bigr).
\]
Thus $w^k$ is obtained by moving the endpoints horizontally, while
$\widehat w^{\,k}$ is obtained by moving them vertically. If
$w\in\zcross\cap\bar H(u)$, every coordinate of
$\widehat w^{\,k}$ belongs to $\bar H(u)$. The
endpoint-switching correspondence of
Lemma~\ref{lemma: k-point paths equivalence}, applied separately to
each block, is weight-preserving and gives
\[
Z_M(w_1^{k_1},\dots,w_m^{k_m})
=
Z_M(\widehat w_1^{\,k_1},\dots,\widehat w_m^{\,k_m}).
\]
Consequently, $Z_M|_{HC(u)}$ is a function of
$Z_M|_{\bar H(u)}$. Moreover, $D(v)\subseteq HC(u)$, since
$v\in\zcross\cap\bar H(u)$. For either of the two environments,
write
\[
V=Z|_{D(v)},\quad
X=Z|_{HC(u)},\quad
Y=Z|_{D(u)},\quad
W=Z|_{\bar V(u)},
\]
where $Z=Z_M$ or $Z=Z_{\tilde M}$, according to the model. The
preceding lemma includes $V$ together with $X$ because the
change-of-measure factor $\rho(V)$ must lie on the same side of the
conditional factorization. In the present application, $V$ is a
deterministic function of $X$, so $\rho(V)$ is
$\sigma(X)$-measurable and the hypotheses
$X\perp\!\!\!\perp W\mid Y$ and
$(V,X)\perp\!\!\!\perp W\mid Y$ are equivalent. The full-space conditional
independence recalled above therefore implies
\[
(V,X)\perp\!\!\!\perp W
\;\mid\;
Y
\]
under the full-space law.

It remains to transfer this relation to the half space. Let $U_0$
be the set of boxes $w\in\zbox$ such that either
$w\in\zcross\cap\bar H(u)$ or $w\in\bar V(u)$. Then
\[
HC(u),\ \bar V(u),\ D(u)
\subseteq
\mathcal E_D(U_0).
\]
Every box in $U_0$ is contained in $u$, and $v$ crosses it
horizontally. Since $u\cap u^\top=v$, we also have
\[
(\sqcup U_0)\cap(\sqcup U_0)^\top \subseteq v.
\]
Theorem~\ref{general bridge}, followed by restriction to the indicated
endpoint sets, shows that the full- and half-space laws of
$(X,Y,W)$ have the same conditional kernel given $V$.

Let $\mu$ and $\tilde\mu$ be the full- and half-space laws of
$V$, respectively.
Lemma~\ref{lemma: conditioning vector absolute continuity}
gives $\tilde\mu\ll\mu$.
Let $\mathbb P$ and $\widetilde{\mathbb P}$ be the full-space
and half-space laws of $(V,X,Y,W)$, respectively.

Write $\rho=d\tilde\mu/d\mu$. Since the two models have the same
conditional kernel given $V$, their joint laws satisfy
\[
\frac{d\widetilde{\mathbb P}}{d\mathbb P}=\rho(V).
\]
This is precisely the density condition in
Lemma~\ref{lemma: conditional independence under marginal change}.
Applying the lemma to the full-space conditional independence above
gives
\[
Z_{\tilde M}|_{HC(u)}
\perp\!\!\!\perp
Z_{\tilde M}|_{\bar V(u)}
\;\mid\;
Z_{\tilde M}|_{D(u)},
\]
as desired.
\end{proof}

\begin{remark}[Extension of the parameter range]
\label{remark: extension of parameter range}
    In Theorems~\ref{inhomogeneous shift and permute} and
    \ref{half-space conditional independence}, we 
    impose the additional requirement on the parameters $\alpha,\alpha'$ that $\alpha_i,\alpha'_i \in (0,1)$ in the geometric case and $\alpha_i,\alpha_i'>0$ in the exponential and log-gamma cases 
    to ensure that the auxiliary full-space environments $M_{\alpha,\alpha}$ and $M_{\alpha',\alpha'}$ are well-defined. 
    
    We note that these restrictions are not necessary for the invariance results to hold. 

    The proofs extend to all admissible half-space parameters (for example, a half-space exponential environment with one $\alpha_i <0 $)
    by the following modification of the auxiliary full-space environment.
    
    To adapt the proofs, let $v=(v^-;v^+)\in\zdiag$ be the
    diagonal square chosen in each proof, and set
    $I=[v^-_1,v^+_1]\cap\Z$.
    We use a constant $B>0$ to make the auxiliary full-space
    weights well-defined while retaining $\alpha_i$ in the first
    parameter sequence for $i\in I$ and in the second for
    $i\notin I$. Namely, replace $M_{\alpha,\alpha}$ by
    $M_{a,b}$, where
    \[
    (a_i,b_i)=
    \begin{cases}
    (\alpha_i,B),&i\in I,\\
    (B,\alpha_i),&i\notin I.
    \end{cases}
    \]
    For example, we can choose
    \[
    B=
    \begin{cases}
    \frac12\min\{c,1\},&\text{in the geometric case},\\
    |c|+1,&\text{in the exponential and log-gamma cases}.
    \end{cases}
    \]
    The half-space assumptions ensure that all weight parameters
    of $M_{a,b}$ are admissible.
    
    The weights indexed by $I\times(\Z\setminus I)$ have the same
    joint law in $M_{a,b}$ and $\tilde M_{\alpha,c}$; these are all
    the weights used outside $v$ in the proofs.
    On $v$, the identity $a_i=\alpha_i$ for $i\in I$ gives
    \eqref{eq: general bridge 1} for $M_{a,b}$ and
    $\tilde M_{\alpha,c}$ by
    Remark~\ref{remark: auxiliary full-space parameters}.
    The proof of
    Lemma~\ref{lemma: conditioning vector absolute continuity}
    also gives
    $\mathcal L(Z_{\tilde M_{\alpha,c}}|_{D(v)})
    \ll\mathcal L(Z_{M_{a,b}}|_{D(v)})$.
    
    For Theorem~\ref{inhomogeneous shift and permute}, define
    $a',b'$ from $\alpha'$ using the same $B$ and $v$.
    For $u\in\{v\}\cup\bigcup_{i=1}^k\overline{U_i}$, with
    $f(v)=v$, we have $a_u=\alpha_u$ and
    $a'_{f(u)}=\alpha'_{f(u)}$. The sequences $\hat b_u$ and
    $\hat b'_{f(u)}$ are obtained from $\hat\alpha_u$ and
    $\hat\alpha'_{f(u)}$, respectively, by replacing the entries
    indexed by $I$ with $B$.
    Since $\alpha_v$ and $\alpha'_v$ are permutations of one another,
    the permutation conditions in
    Proposition~\ref{duncan parameter shift} are preserved.
    The conditional-law identity and absolute continuity also hold
    with $(a,b,\alpha)$ replaced by $(a',b',\alpha')$.
    Thus the two proofs adapt to these auxiliary full-space environments.
    We impose the additional restrictions above to avoid introducing
    these more general formulations of the transfer theorems.
\end{remark}

\subsection{Proofs of the introductory formulations}
\label{subsec: intro proofs}

For completeness, we include proofs of the introductory results stated in Subsection~\ref{subsec: main results} from the more general invariance theorems
of Section~\ref{main section}.

For $u=(u^-_1,u^-_2;u^+_1,u^+_2)\in\zcross$, write
\[
    u^\boxslash
    \defeq (u^-_1,u^-_1;u^+_1,u^+_1)\in\zdiag .
\]
Then $(u^\boxslash,u)\in\mathcal H$. 

\begin{proof}[Proof of Theorem~\ref{multi-point shift intro}]
Apply Theorem~\ref{tower shift} under the case that there is only $U=U_1 = \{u_1,\dots, u_k\}$
present in the tower, with corresponding $U' = U'_1 = \{u'_1,\dots, u'_k\}$ such that $U' = T_cU$ for some $c\in \Z^2$.
The assumed box $v$ verifies the anchoring condition in
Theorem~\ref{tower shift}.
The result states that the law of
$
Z_{\tilde M}|_{\mathcal E_D(U)}
$
is preserved under the shift $T_c$.
Included in it are the multi-point passage values across all vectors of the form
\[
(u_{j_1}^{s_1},u_{j_2}^{s_2}, \dots, u_{j_t}^{s_t}).
\]
By hypothesis, $\mathbf u\in\mathcal E_D(U)$. Since a common
translation preserves disjointness, $\mathbf u'=T_c\mathbf u$
belongs to $\mathcal E_D(U')$. Moreover, $(u_i)\in\mathcal E_D(U)$
for every $i$. The result follows by restricting the preceding
distributional identity.
\end{proof}

\begin{proof}[Proof of Theorem~\ref{intro tower shift}]
By applying a common diagonal translation to all $u_i'$, we may assume that $c_1=(0,h)$. 
Then
set $v = u_1^\boxslash$, $U_i = \{u_i\}, U_i'=\{u'_i\}$. Noting that
$(v,u_1),(v,u'_1)\in \mathcal H$ and the remaining conditions are satisfied by assumption, apply Theorem~\ref{tower shift} and take the marginal indexed by $u_1,\dots, u_k$ to get the result.  
\end{proof}

\begin{proof}[Proof of Theorem~\ref{intro disjoint shift}]
Take $V  =\{v\}, U_i = \{u_i\}$ and $U'_i = \{u'_i\}$. 
Since $v \in \zcross$, we have $(v^\boxslash,v) \in \mathcal H$. 
Moreover, $(v,u_i) \in \mathcal H$ implies $(v^\boxslash,u_i)\in\mathcal H$, and similarly
$(v,u_i')\in\mathcal H$ implies $(v^\boxslash,u_i')\in\mathcal H$. Hence the
hypotheses of Theorem~\ref{disjoint shift} are precisely the assumptions of
Theorem~\ref{intro disjoint shift}, in the special case of singleton sets. 
Taking the
marginal indexed by $v,u_1,\dots,u_k$ gives the result. 
\end{proof}

\begin{proof}[Proof of Theorem~\ref{intro inhomogeneous shift}]
Apply Theorem~\ref{inhomogeneous shift and permute} with $U_i=\{u_i\}$ and $U_i'=\{u_i'\}$.
The crossing and shift hypotheses are exactly those of Theorem~\ref{intro tower shift}.
For a singleton $U_i=\{u_i\}$, the completion $\overline{U_i}$ is just $U_i$ itself, 
and
the parameter sequences $\alpha_{u_i}$ and $\hat\alpha_{u_i}$ are respectively $(\alpha_j:j\in I(u_i))$ and $(\alpha_j:j\in J(u_i))$.
Thus the parameter-permutation assumption in
Theorem~\ref{inhomogeneous shift and permute} is exactly the hypothesis imposed in
Theorem~\ref{intro inhomogeneous shift}. 
Taking the marginal indexed by
$u_1,\dots,u_k$ gives
the result.
\end{proof}

\appendix
\section{RSK growth diagrams and Sch\"utzenberger's lemma}
\label{appendix}

We will adopt the definition of the RSK correspondence using row insertions, as described in \cite[Section~7.11]{stanley},
derive a growth diagram formulation of the RSK from it, and finally use it to prove Sch\"utzenberger's lemma (Proposition~\ref{schutzenberger}). 

We fix $m,n \in \N$ and let the matrix $A = (a_{ij}) \in (\Z_{\geq 0})^{m\times n}$ be an $m\times n$ matrix of nonnegative integer entries. 
From the matrix, we can obtain a \emph{two-line array} of the following form
\begin{gather*}
    w_A = 
    \begin{pmatrix}
    i_1 & i_2 & \dots & i_N \\
    j_1 & j_2 & \dots & j_N
    \end{pmatrix}
\end{gather*}
where $N = \sum_{i,j}a_{i,j}$,
and the pairs $(i_t,j_t)$ are ordered lexicographically, meaning that $i_1\leq i_2\leq \dots \leq i_N$ and whenever $i_r = i_s$ for $r\leq s$, we have $j_r \leq j_s$.
For each $1\leq i\leq m, 1\leq j\leq n$, the pair $(i,j)$ appears exactly $a_{ij}$ times in this array. 

For $P=(P_{ij})$ a semistandard Young tableau (SSYT) and $k$ a positive integer, the \emph{row insertion} $P \gets k$ is defined as follows:

First, insert $k$ into the first row of $P$. If every entry in the
first row is at most $k$, then place $k$ at the end of the first row and terminate the procedure.
Otherwise, let $l$ be the leftmost entry in the first row which is strictly larger than $k$. Replace
this entry $l$ by $k$, and then insert $l$ into the second row by the same rule. 
That is, if every
entry in the second row is at most $l$, place $l$ at the end of the second row and terminate;
otherwise, replace the leftmost entry strictly larger than $l$ by $l$, and insert the bumped entry
into the next row. 
Continuing in this way, the procedure eventually terminates by adding one new
entry to the end of some (possibly empty) row. The resulting tableau is denoted $P\leftarrow k$.

We denote the box in the $s$-th position of the $r$-th row in a Young tableau by the pair $(r,s)$.
Now, the pair $(P,Q)$ of SSYTs associated with $A$ is obtained from $w_A$ using the following procedure:

Start with two empty tableaux $(P(0),Q(0))=(\emptyset,\emptyset)$.
Suppose that after $t-1$ steps we have constructed a pair of SSYTs
$(P(t-1),Q(t-1))$ of the same shape.
To perform the $t$-th step, insert $j_t$ into
$P(t-1)$ by row insertion:
\[
P(t)=P(t-1)\leftarrow j_t.
\]
This adds exactly one new box to $P(t-1)$. Let $(r,s)$ be the position of this new box. Define
$Q(t)$ to be the tableau obtained from $Q(t-1)$ by adding a new box in the same position
$(r,s)$, filled with the entry $i_t$. After all entries of the two-line array have been inserted,
we set
\[
(P,Q)=(P(N),Q(N)),
\]
where $N=\sum_{i,j}a_{ij}$ is the total number of entries in the two-line array.

We write
\[
(P(t+k),Q(t+k)) = 
(P(t),Q(t)) \gets 
 \begin{pmatrix}
    i_{t+1} & i_{t+2} & \dots & i_{t+k} \\
    j_{t+1} & j_{t+2} & \dots & j_{t+k}
    \end{pmatrix}
\]
for the process of obtaining $(P(t+k),Q(t+k))$ from $(P(t),Q(t))$ under the above procedure. 

Denote by $I(P \gets j)$ the insertion path of inserting the number $j$ into the SSYT $P$, which is the set of boxes being added or having the contained number changed during the insertion.
We note the following property of the insertion path:

\begin{lemma}
[{\cite[Lemma~7.11.2(b)]{stanley}}]
\label{appendix-insertion path lemma}
    Let $P$ be an SSYT, and let $j \leq k$. Then $I(P\gets j)$ lies strictly to the left of $I((P \gets j)\gets k)$. More precisely, if $(r,s) \in I(P\gets j)$ and $(r,t) \in I((P\gets j)\gets k)$, then $s<t$. Moreover, $I((P\gets j)\gets k)$ does not extend below the bottom of $I(P\gets j)$. 
\end{lemma}

We will only use it through the next corollary:

\begin{corollary}\label{appendix-insertion order corollary}
    Consider $P \gets (j_1, j_2,\dots j_n)$ where $j_1\leq j_2\leq \dots \leq j_n$. Let $(r_k,s_k)$ be the box added during the insertion of $j_k$, for each $1\leq k\leq n$, regardless of what number it contains.
    Then $(r_1,s_1),(r_2,s_2),\dots,(r_n,s_n)$ goes strictly from left to right and weakly from bottom to top.
    More precisely, if $k < k'$, then $r_k \geq r_{k'}$ and $s_k < s_{k'}$.
\end{corollary}
\begin{proof}
    The result follows by repeated application of Lemma~\ref{appendix-insertion path lemma} above.
\end{proof}

The symmetry property of the RSK is also important:

\begin{theorem}
[{\cite[Theorem~7.13.1]{stanley}}]
\label{appendix-symmetry RSK}
    If the matrix $A$ corresponds to $(P,Q)$ under the RSK, then its transpose $A^\top$ corresponds to $(Q,P)$. 
\end{theorem}

Now let $(c_i) \in (\Z_{\geq 0})^m$ and $(r_i) \in (\Z_{\geq 0})^n$. 
Define the matrices
$A' = (a'_{ij}) \in (\Z_{\geq 0})^{(m+1)\times n}$ by $a'_{ij} = a_{ij}$ for $1\leq i \leq m, 1\leq j\leq n$ and $a'_{m+1,j} = r_j$ for $1\leq j \leq n$,
 and $A'' = (a''_{ij}) \in (\Z_{\geq 0})^{m\times (n+1)}$ by 
 $a''_{ij} = a_{ij}$ for $1\leq i\leq m, 1\leq j\leq n$ and $a''_{i,n+1} = c_i$ for $1 \leq i \leq m$. 
 Let $d \in \Z_{\geq 0}$, define also the matrix 
 $A''' = (a'''_{ij})\in (\Z_{\geq 0})^{(m+1)\times (n+1)}$ by
 $a'''_{ij} = a_{ij}$ for $1\leq i \leq m$, $1\leq j\leq n$, $a'''_{m+1,j} = r_j$ for $1\leq j \leq n$, $a'''_{i,n+1} = c_i$ for $1 \leq i \leq m$, and $a'''_{m+1, n+1} = d$. 
 In words, $A'$ is obtained from $A$ by adding $(r_i)$ as the last row, $A''$ is obtained from $A$ by adding $(c_i)$ as the last column, and $A'''$ is obtained from $A$ by adding both the last row and the column, and setting the additional entry $a'''_{m+1, n+1}$ to $d$.

 Suppose that the RSK maps $A,A',A'',A'''$ to $(P,Q),(P',Q'), (P'',Q''),(P''',Q''')$, respectively. 

 We can see that since the added entries in $A'$ all have first coordinate as $m+1$, its corresponding two-line array $w_{A'}$ is of the form
 \begin{gather*}
     w_A + 
     \begin{pmatrix}
         m+1 & m+1 & \dots & m+1 \\
         j'_1 & j'_2 & \dots & j'_{k'}
     \end{pmatrix}
 \end{gather*}
 where the plus sign indicates concatenation of two-line arrays, $k' = \sum_i r_i$, and
 $r_i$ is the number of times $i$ appears among $(j'_1,\dots, j'_{k'})$ for each $i$
Hence 
\[
(P',Q')  = (P,Q) \gets \begin{pmatrix}
         m+1 & m+1 & \dots & m+1 \\
         j'_1 & j'_2 & \dots & j'_{k'}
     \end{pmatrix}.
\]
In particular,
\[
P' = P \gets (j'_1,\dots j'_{k'})
\]
and 
$Q'$ is obtained from $Q$ by adding boxes containing $m+1$ on the right of each row.

Symmetrically, 
$w_{(A'')^\top}$ is of the form
\[
w_{A^\top} + 
\begin{pmatrix}
     n+1 & n+1 & \dots & n+1 \\
     j''_1 & j''_2 & \dots & j''_{k''}
\end{pmatrix}
\]
where $c_i$ is the number of times $i$ appears among $(j''_1,\dots, j''_{k''})$ for each $i$ and $k'' = \sum_i c_i$. 
Hence by Theorem~\ref{appendix-symmetry RSK},
\[
(Q'',P'')  = (Q,P) \gets 
    \begin{pmatrix}
         n+1 & n+1 & \dots & n+1 \\
        j''_1 & j''_2 & \dots & j''_{k''}
     \end{pmatrix}.
\]
In particular,
$P''$ is obtained from $P$ by adding boxes containing $n+1$ on the right of each row.

Also note that since $A'''$ is obtained from $A''$ by adding a row, 
\[
w_{A'''} = w_{A''} + 
\begin{pmatrix}
    m+1 & m+1 & \dots & m+1 \\
    j'_1 & j'_2 & \dots & j'_{k'}
\end{pmatrix}
+
\begin{pmatrix}
    m+1 \\
    n+1
\end{pmatrix}
\times d
\]
where the middle term corresponds to $(r_i)$ as above and the last term is the two-line array with $(m+1,n+1)$ repeated $d = a'''_{m+1,n+1}$ times. 
Hence,
\[
(P''',Q''') = (P'',Q'') \gets 
\begin{pmatrix}
    m+1 & m+1 & \dots & m+1 \\
    j'_1 & j'_2 & \dots & j'_{k'}
\end{pmatrix}
+
\begin{pmatrix}
    m+1 \\
    n+1
\end{pmatrix}
\times d
,
\]
and in particular
\[ P''' = 
\left( P'' \gets 
(j'_1,j'_2,\dots, j'_{k'})
\right)
\gets 
(n+1)
\times d.
\]

\begin{lemma}\label{appendix-lemma for s'_i}
    Suppose $Y$ is a SSYT with all entries less than $u$, and $Y\cup s$ is an SSYT obtained from $Y$ by adding $s_i\in \Z_{\geq 0}$ boxes containing the number $u$ to the end of $i$-th row of $Y$, for each $i$.  
    Let $w = (t_1, t_2,\dots, t_r)$ be a weakly increasing sequence of positive integers with each $t_i < u$.
    Let $\hat{Y} = Y \gets w$ be the result of inserting $w$ into $Y$.
    Then $(Y \cup s \gets w)$, the result of inserting $w$ into $Y\cup s$, is of the form $\hat{Y} \cup s'$, meaning it can be obtained from $\hat{Y}$ by adding $s'_i$ boxes containing $u$ to the end of its $i$-th row. 
    Denoting the number of boxes in the $i$-th row of $Y$ by $|Y_i|$,
    and adopting the conventions $s_0=0$ and $|Y_0| = |\hat Y_0| =0$,
    we have that for each $i$,
    \begin{equation} 
    \label{eq: appendix insertion relation} 
    s'_i = \max(0,s_i -( |\hat{Y}_i| -|Y_i|))
    + \min(s_{i-1}, |\hat{Y}_{i-1}| -|Y_{i-1}|).
    \end{equation}
\end{lemma}
\begin{proof}
    We proceed by induction on $r$, the length of $w$. If $r=0$, no insertion occurs and $\hat{Y} = Y$, so the identities hold.  

    Assuming the statement of the lemma as inductive hypothesis, 
    consider next the insertion of $t_{r+1} = t<u$ into $\hat{Y} = Y \gets w$, where $t \geq t_1,t_2,\dots t_r$.
    If $t$ is not placed directly at the end of the first row, it bumps an element $y_1 > t$ from that row; this element then bumps an element $y_2 > y_1$ from the second row, and so on, until an element $y_n$ is placed at the end of row $n+1$. 
    
    To unify notation, we can think of the case $n=0$ as placing $y_0 = t$ directly at the end of the first row of $Y$.
    By Corollary~\ref{appendix-insertion order corollary},
    it must be that the previous process of inserting $w = (t_1,\dots, t_r)$ into $Y$ has not added any boxes to the first $n$ rows of $Y$, so that $|\hat{Y}_i| = |Y_i|$ for each $1\leq i \leq n$. 

    Now consider inserting $t_{r+1}$ into $\hat{Y} \cup s'$. 
    Since the bumped elements $y_i$ are less than $u$ for each $0\leq i\leq n$, and $\hat{Y} \cup s'$ only differs from $\hat{Y}$ by having additional boxes containing $u$ on the right, 
    the same bumping process will happen as above, until $y_n$ gets bumped off the $n$-th row of $\hat{Y}\cup s'$. 
    After this, we split into two cases. 
    For each case, we will show that $(\hat{Y} \cup s') \gets t$ can be written as $(\hat{Y} \gets t) \cup s''$ for some $s''$ which records the new way the $u$'s appear on the end of each row, and then verify
    \begin{equation} \label{eq: appendix to verify}
        s_i'' = \max(0,s_i-( |(\hat{Y} \gets t)_i| - |Y_i|)) 
        + \min(s_{i-1},|(\hat{Y}\gets t)_{i-1}| - |{Y}_{i-1}|)
    \end{equation}
    for each $i$ to conclude the induction. 
    
    \textbf{Case 1}: $s'_{n+1} = 0$.
    Then the $(n+1)$-th row of $\hat{Y} \cup s'$ is identical to the $(n+1)$-th row of $\hat{Y}$, so $y_n$ will be placed at its end and the insertion process is done. 
    As no entry equal to $u$ has been bumped, we have 
    \[ 
    (\hat{Y} \cup s') \gets t = (\hat{Y} \gets t) \cup s'' 
    \]
    for $s'' = s'$. 
    Note that $|(\hat{Y}\gets t)_i| = |\hat{Y}_i|$ for $i\neq n+1$, so \eqref{eq: appendix to verify} is automatically satisfied for $i\neq n+1,n+2$,  and we have that $|(\hat{Y}\gets t)_{n+1}| = |\hat{Y}_{n+1}| + 1.$ 
    
    From the induction hypothesis \eqref{eq: appendix insertion relation} for $s'_{n+1}$, 
    \[ 
    0 = s'_{n+1}  = \max(0,s_{n+1} -( |\hat{Y}_{n+1}| -|Y_{n+1}|))
    + \min(s_{n}, |\hat{Y}_{n}| -|Y_{n}|).
    \]
    As $|\hat{Y}_{n}| =|Y_{n}|$, we have $\min(s_n,|\hat{Y}_{n}| -|Y_{n}|) = 0$, 
    so it must be that $s_{n+1} \leq ( |\hat{Y}_{n+1}| -|Y_{n+1}|)$,
    which means 
    $s_{n+1} - (|\hat{Y}_{n+1}| +1-|Y_{n+1}|) < 0$ and
    $\min(s_{n+1}, |\hat{Y}_{n+1}| -|Y_{n+1}|) = s_{n+1} = \min(s_{n+1}, |\hat{Y}_{n+1}| + 1 -|Y_{n+1}|)$.
    Consequently 
    \begin{align*} 
    &\phantom{={}}
     \max(0,s_{n+1}-( |(\hat{Y} \gets t)_{n+1}| - |Y_{n+1}|)) 
    + \min(s_{n},|(\hat{Y}\gets t)_{n}| - |{Y}_{n}|) \\
    &=
    \max(0,s_{n+1} -( |\hat{Y}_{n+1}|+1 -|Y_{n+1}|))
    + \min(s_{n}, |\hat{Y}_{n}| -|Y_{n}|) \\
    &= 0 \\
    &=s'_{n+1},
    \end{align*}
    so \eqref{eq: appendix to verify} is satisfied for $i=n+1$, and
    \begin{align*}
        &\phantom{=\ } \max(0,s_{n+2}-( |(\hat{Y} \gets t)_{n+2}| - |Y_{n+2}|)) 
        + \min(s_{n+1},|(\hat{Y}\gets t)_{n+1}| - |{Y}_{n+1}|) \\
      &= \max(0,s_{n+2}-( |\hat{Y}_{n+2}| - |Y_{n+2}|)) 
     + \min(s_{n+1},|\hat{Y}_{n+1}| + 1 - |{Y}_{n+1}|)\\
     &= \max(0,s_{n+2}-( |\hat{Y}_{n+2}| - |Y_{n+2}|)) 
     + \min(s_{n+1},|\hat{Y}_{n+1}| - |{Y}_{n+1}|)\\
     &= s'_{n+2},
    \end{align*}
    so \eqref{eq: appendix to verify} is satisfied for $i = n+2$. 

    \textbf{Case 2:} $s'_{n+1} > 0$. 
    Then the $(n+1)$-th row of $\hat{Y} \cup s'$ is the same row of $\hat{Y}$ with some additional $u$'s at the end. 
    Hence placing $y_n < u$ in would bump the first $u$ off, which then gets placed at the end of the $(n+2)$-th row (since there are no entries larger than $u$). 
    No other entry equal to $u$ has been bumped, so we have 
    \[ 
    (\hat{Y} \cup s') \gets t = (\hat{Y} \gets t) \cup s''
    \]
     where
     $s''_{n+1} = s'_{n+1}-1$, 
    $s''_{n+2} = s'_{n+2}+1$,
    and $s''_i = s'_i$ for $i \neq n+1,n+2$. 
    As in the previous case, \eqref{eq: appendix to verify} is automatically satisfied for $i \neq n+1, n+2$. 

    From the induction hypothesis \eqref{eq: appendix insertion relation}, 
    \[ 
    1 \leq s'_{n+1}  = \max(0,s_{n+1} -( |\hat{Y}_{n+1}| -|Y_{n+1}|))
    + \min(s_{n}, |\hat{Y}_{n}| -|Y_{n}|).
    \]
    As $|\hat{Y}_{n}| =|Y_{n}|$, we have $\min(s_n,|\hat{Y}_{n}| -|Y_{n}|) = 0$, 
    so it must be that $s_{n+1} \geq  |\hat{Y}_{n+1}| -|Y_{n+1}|+1$,
    which means 
    $s_{n+1} - (|\hat{Y}_{n+1}| +1-|Y_{n+1}|) \geq 0$,
    $\min(s_{n+1}, |\hat{Y}_{n+1}| -|Y_{n+1}|) = |\hat{Y}_{n+1}| -|Y_{n+1}|$, 
    and 
    $\min(s_{n+1}, |\hat{Y}_{n+1}| + 1 -|Y_{n+1}|) = |\hat{Y}_{n+1}| + 1 -|Y_{n+1}|$.
    Consequently 
    \begin{align*} 
    &\phantom{={}}
     \max(0,s_{n+1}-( |(\hat{Y} \gets t)_{n+1}| - |Y_{n+1}|)) 
    + \min(s_{n},|(\hat{Y}\gets t)_{n}| - |{Y}_{n}|) \\
    &= \max(0,s_{n+1}-( |\hat{Y}_{n+1}| + 1 - |Y_{n+1}|)) 
    + \min(s_{n},|\hat{Y}_{n}| - |{Y}_{n}|) \\
    &=
    s_{n+1} -( |\hat{Y}_{n+1}|+1 -|Y_{n+1}|)
    + \min(s_{n}, |\hat{Y}_{n}| -|Y_{n}|) \\
    &= \max(0,s_{n+1}-( |\hat{Y}_{n+1}| - |Y_{n+1}|)) 
    + \min(s_{n},|\hat{Y}_{n}| - |{Y}_{n}|) -1\\
    &=s'_{n+1} - 1 \\
    &= s''_{n+1},
    \end{align*}
    so \eqref{eq: appendix to verify} is satisfied for $i=n+1$, and
    \begin{align*}
     &\phantom{=\ } \max(0,s_{n+2}-( |(\hat{Y} \gets t)_{n+2}| - |Y_{n+2}|)) 
        + \min(s_{n+1},|(\hat{Y}\gets t)_{n+1}| - |{Y}_{n+1}|) \\
      &= \max(0,s_{n+2}-( |\hat{Y}_{n+2}| - |Y_{n+2}|)) 
     + \min(s_{n+1},|\hat{Y}_{n+1}| + 1 - |{Y}_{n+1}|)\\
      &= \max(0,s_{n+2}-( |\hat{Y}_{n+2}| - |Y_{n+2}|)) 
     + |\hat{Y}_{n+1}| - |{Y}_{n+1}| + 1\\
     &= \max(0,s_{n+2}-( |\hat{Y}_{n+2}| - |Y_{n+2}|)) 
     + \min(s_{n+1},|\hat{Y}_{n+1}| - |{Y}_{n+1}|) + 1\\
     &= s'_{n+2} + 1 \\
     &= s''_{n+2},
    \end{align*}
    so \eqref{eq: appendix to verify} is satisfied for $i = n+2$. 
\end{proof}

Suppose $s$ is defined such that $P''$ is obtained from $P$ by adding $s_i$ boxes containing $(n+1)$ to its $i$-th row for each $i$. 
Then applying the lemma with $Y = P$, $Y \cup s = P''$, $u = n+1$, and $w = (j'_1,j'_2,\dots, j'_{k'})$,
so that $\hat{Y} = P'$ we get that 
\[
P'' \gets (j'_1,\dots j'_{k'}) = (P \gets (j'_1,\dots ,j'_{k'})) \cup s' = P' \cup s',
\]
where 
\[
s'_i = \max(0,s_i-(|P'_i|-|P_i|)) + \min(s_{i-1}, |P'_{i-1}|-|P_{i-1}|). 
\]
Since all numbers contained in $P'\cup s'$ are at most $(n+1)$, inserting a new sequence of $(n+1)$'s into it will only have them added to the end of the first row. 
Recalling that $P''' = (P''\gets (j'_1,\dots, j'_{k'}))\gets (n+1)\times d$, we thus have for each $i$ that
\[
|P'''_i| = |P'_i| + s'_i + d\mathds{1}_{i=1}.
\]

Substituting the expression for $s'_i$, writing $s_i = |P''_i| - |P_i|$, and simplifying gives
\begin{equation}
\label{eq: appendix recursive relation P'''_i}
    |P'''_i| = \max(|P'_i|,|P''_i|) + \min(|P'_{i-1}|,|P''_{i-1}|) - |P_{i-1}| + d \mathds{1}_{i=1}
\end{equation}
for each $i \geq 1$, where $|P_0| = |P'_0|=|P''_0|=0$.
This will be the recursive relation in the growth diagram of the RSK, which can be used as an alternative definition of the RSK.

\begin{theorem}
\label{appendix-equivalence theorem}
    Suppose $A = (a_{ij})_{i,j\geq 1}$ is an infinite matrix with entries in $\Z_{\geq 0}$. 
    For $r,s \geq 0$, write 
    \[
    A^{(r,s)} = (a_{ij})_{1\leq i\leq r, \,1\leq j \leq s},
    \]
    with $A^{(0,s)}$ and $A^{(r,0)}$ interpreted as empty matrices. 
    For every $r,s \in \Z_{\geq 0}$, define a partition $\lambda^{(r,s)}$ recursively as follows:
    \begin{enumerate}
        \item    
        Let each $\lambda^{(r,0)}$ and $\lambda^{(0,s)}$ be the empty partition $(0,0,\dots)$.
        \item 
        For $r,s \geq 1$, once $\lambda^{(r-1,s-1)}$, $\lambda^{(r-1,s)}$, and $\lambda^{(r,s-1)}$ have been defined, 
        let $\lambda^{(r,s)}$ be the partition with components
        \[\lambda^{(r,s)}_k = \max(\lambda^{(r-1,s)}_k, \lambda^{(r,s-1)}_k) + \min(\lambda^{(r-1,s)}_{k-1}, \lambda^{(r,s-1)}_{k-1})-\lambda^{(r-1,s-1)}_{k-1}
        + a_{rs} \mathds{1}_{k=1}\] 
        for each $k \geq 1$, with the convention that every partition has zeroth component 0. 
    \end{enumerate}
    Then,
    $\lambda^{(r,s)}$ is the RSK shape of $A^{(r,s)}$ for every $r,s \geq 0$. In particular, for fixed $m,n \geq 1$, the RSK maps the  $m \times n$ truncation $A^{(m,n)}$ to a pair of SSYTs $(P,Q)$ 
    with $\lambda^{(m,n)}$ as their common shape. 
    Moreover, $P$ can be reconstructed from the sequence of Young diagrams
    \[
    \emptyset = \lambda^{(m,0)} \subseteq \lambda^{(m,1)} \subseteq \dots \subseteq \lambda^{(m,n)}
    \]
    where $\lambda^{(m,s)}$ is the shape of the subtableau of $P$ consisting of entries at most $s$. 
    Similarly, $Q$ can be reconstructed from the sequence of Young diagrams 
    \[
    \emptyset = \lambda^{(0,n)} \subseteq \lambda^{(1,n)} \subseteq \dots\subseteq \lambda^{(m,n)}, 
    \]
    where $\lambda^{(r,n)}$ is the shape of the subtableau of $Q$ consisting of entries at most $r$.
    
\end{theorem}
\begin{proof}
    We prove the shape assertion by induction on $r+s$. The cases $r=0$ or $s=0$ are immediate. If $r=1$ or $s=1$, both the recurrence and RSK give the one-row partition whose first part is the sum of the entries of $A^{(r,s)}$. Now let $r,s \geq 2$.
    By induction, the RSK shapes of $A^{(r-1,s-1)}$, $A^{(r,s-1)}$, and $A^{(r-1,s)}$ are the corresponding partitions $\lambda$. Apply \eqref{eq: appendix recursive relation P'''_i} with 
    the three matrices being $A,A',A''$ respectively, $A'''=A^{(r,s)}$ and $d =a_{rs}$, 
    the resulting identity is exactly the defining recurrence for $\lambda^{(r,s)}$, proving the induction step. 

    For the reconstruction statement, fix $m,n$. 
    After the first $r$ rows of $A^{(m,n)}$ have been processed, the recording tableau has shape $\lambda^{(r,n)}$, and its newly added boxes contain $r$. 
    Hence the chain $\lambda^{(0,n)} \subseteq \dots \subseteq \lambda^{(m,n)}$ reconstructs $Q$.
    Applying the same argument to $(A^{(m,n)})^\top$ and using Theorem~\ref{appendix-symmetry RSK} shows that $\lambda^{(m,0)} \subseteq \dots \subseteq \lambda^{(m,n)}$ reconstructs $P$. 
\end{proof}

Using this growth diagram formulation of the RSK correspondence, we are ready to prove 
Sch\"utzenberger's lemma by induction. 

\begin{proof}
    [Proof of Proposition~\ref{schutzenberger}]
    Suppose that in Theorem~\ref{appendix-equivalence theorem} the matrix $A$ is symmetric, i.e. $a_{ij} = a_{ji}$ for all $i\neq j$. 
    From the symmetry of the definition, we have $\lambda^{(i,j)} = \lambda^{(j,i)}$ for each $i\neq j$. 
    Recall that for a partition $\lambda$, $O(\lambda)$ is defined to be its alternating sum $\sum_i (-1)^{i-1} \lambda_i $. 
    We want to show that in this case 
    \[
    \sum_{i=1}^n a_{ii} = O(\lambda^{(n,n)})
    \]
    for each $n\geq 1$.
    For this, we induct on $n$. 
    The case $n=1$ is trivial, since $\lambda^{(1,1)} = (a_{11},0,0,\dots)$. 
    Then, assume that $\sum_{i=1}^{n-1}a_{ii}=O(\lambda^{(n-1,n-1)})$.
    By the recursive relation of the growth diagram, for each $i\geq 1$, we have
    \[
    \lambda^{(n,n)}_i
    =
    \lambda^{(n,n-1)}_i
    +
    \lambda^{(n,n-1)}_{i-1}
    -
    \lambda^{(n-1,n-1)}_{i-1}
    +
    a_{nn}\mathds{1}_{i=1}.
    \]
    Here we used the symmetry of $A$, which gives
    $\lambda^{(n,n-1)}=\lambda^{(n-1,n)}$.
    Therefore,
    \begin{align*}
        O(\lambda^{(n,n)})
        &=
        \sum_{i}(-1)^{i-1}\lambda^{(n,n)}_i \\
        &=
        a_{nn}
        +
        \sum_{i}(-1)^{i-1}\lambda^{(n,n-1)}_i
        +
        \sum_{i}(-1)^{i-1}\lambda^{(n,n-1)}_{i-1}
        -
        \sum_{i}(-1)^{i-1}\lambda^{(n-1,n-1)}_{i-1} \\
        &=
        a_{nn}
        +
        O(\lambda^{(n,n-1)})
        -
        O(\lambda^{(n,n-1)})
        +
        O(\lambda^{(n-1,n-1)}) \\
        &=
        a_{nn}+O(\lambda^{(n-1,n-1)})\\
        &= \sum_{i=1}^n a_{ii},
    \end{align*}
    where the last step comes from the induction hypothesis. This completes the proof. 
\end{proof}

\printbibliography

\end{document}